\documentclass[10pt,onecolumn,letterpaper]{article}
\usepackage[authoryear,sort]{natbib}
\usepackage{filecontents}
\usepackage{graphicx}
\usepackage{amssymb}
\usepackage[pdfpagelabels=true,plainpages=false,colorlinks=true,
linkcolor=blue,citecolor=blue,urlcolor=blue]{hyperref}
\usepackage[latin1]{inputenc}
\usepackage{algorithm}
\usepackage{algorithmic}
\usepackage{amsfonts}
\usepackage{amsmath}
\usepackage{cite}
\newcommand{\be}{\begin{equation}}
	\newcommand{\ee}{\end{equation}}
\newcommand{\bea}{\begin{eqnarray}}
	\newcommand{\eea}{\end{eqnarray}}
\usepackage{caption}
\usepackage{subcaption}

\usepackage{graphicx,epsfig}
\usepackage{enumerate}
\usepackage{setspace}
\usepackage{amssymb}
\usepackage{amsmath}
\usepackage{amsthm}
\usepackage{appendix}
\newtheorem{alg1}{\textbf{Algorithm}}[section]
\numberwithin{equation}{section}
\newtheorem{thm}{Theorem}[section]
\theoremstyle{definition}
\newtheorem{dfn}{Definition}[section]
\theoremstyle{Remark}
\theoremstyle{lemma}
\newtheorem{rem}{Remark}[section]
\newtheorem{lem}{Lemma}[section]

\theoremstyle{example}
\newtheorem{example}{Example}[section]
\begin{document}

\title{A Quasi Newton Method for Uncertain Multiobjective Optimization Problems via Robust Optimization Approach}
\author{Shubham Kumar\footnote{IIITDM Jabalpur, India, E-mail: kumarshubham3394@gmail.com},
Nihar Kumar Mahato \footnote{IIITDM Jabalpur, India, E-mail: nihar@iiitdmj.ac.in, Corresponding Author},
Md Abu T Ansary \footnote{IIT Jodhpur, India, E-mail: md.abutalha2009@gmail.com},
Debdas Ghosh \footnote{IIT(BHU), India, E-mail: debdas.mat@iitbhu.ac.in}\\}

\date{}
\maketitle

\begin{abstract}
	{In this paper, we propose a quasi Newton method to solve the robust counterpart of an uncertain multiobjective optimization problem under an arbitrary finite uncertainty set. Here the robust counterpart of an uncertain multiobjective optimization problem is the minimum of objective-wise worst case, which is a nonsmooth deterministic multiobjective optimization problem. In order to solve this robust counterpart with the help of quasi Newton method, we construct a sub-problem using Hessian approximation and solve it to determine a descent direction for the robust counterpart. We introduce an Armijo-type inexact line search technique to find an appropriate step length, and develop a modified BFGS formula to ensure positive definiteness of the Hessian matrix at each iteration. By incorporating descent direction, step length size, and modified BFGS formula, we write the quasi Newton's descent algorithm for the robust counterpart. We prove the convergence of the algorithm under standard assumptions and demonstrate that it achieves superlinear convergence rate. Furthermore, we validate the algorithm by comparing it with the weighted sum method through some numerical examples by using a performance profile.}
	\end{abstract}

\section{Introduction}
Before introducing an uncertain multiobjective optimization problem, we start with a multiobjective optimization problem. In multiobjective optimization problem, we want to solve more than one objective functions simultaneously. But, there is no single point which minimize all objective function at once, so the concept of optimality is replaced by Pareto optimality for multiobjective optimization problem. Multiobjective optimization problem have many application in field of science, for that see \citet{bhaskar2000applications,stewart2008real} and references therein.
Mathematically, a deterministic multiobjective optimization problem is given as follows
\begin{equation}\label{0.1}	P:~~~ \displaystyle \min_{x\in D\subset\mathbb{R}^n}\zeta(x)
\end{equation}
where $\zeta:\mathbb{R}^n\to\mathbb{R}^m$ such that $\zeta(x)=(\zeta_1(x),\zeta_2(x),\ldots,\zeta_m(x))$ and  $\zeta_j:\mathbb{R}^n\to\mathbb{R}$, $j\in\Lambda=\{1,2,\ldots,m\}$, $D\subset\mathbb{R}^n$ is a feasible set. In particular, if $D=\mathbb{R}^n$ then the above deterministic  multiobjective  optimization problem is called unconstrained multiobjective optimization problem.
\par Various solution strategies for multiobjective optimization problems have been proposed in the literature ~\citet{fliege2000steepest,drummond2005steepest,fliege2009newton,miettinen1999nonlinear,ansary2020sequential,ansary2019sequential,fukuda2014survey,ansary2015modified,lai2020q,ansary2021globally,gass1955computational,povalej2014quasi}.
%\citet{,4,10,13,17,18,19,20,22,25}.
Several iterative methods originally designed for scalar optimization problems have been extended and analyzed to address multiobjective optimization problems. The field of iterative methods for multiobjective optimization problems was initially introduced by \citet{fliege2000steepest}. Since then, several authors have expanded upon this area, including the development of Newton's method (\citet{fliege2009newton}), quasi Newton method (\citet{ansary2015modified,povalej2014quasi,lai2020q,mahdavi2020superlinearly,morovati2018quasi,qu2011quasi}), conjugate gradient method (\citet{gonccalves2020extension,lucambio2018nonlinear}), projected gradient method (\citet{cruz2011convergence,fazzio2019convergence,fukuda2013inexact,fukuda2011convergence,drummond2004projected}), and proximal gradient method (\citet{bonnel2005proximal,ceng2010hybrid}). Convergence properties are a common characteristic of these methods. Notably, these approaches do not require transforming the problem into a parameterized form, distinguishing them from scalarization techniques given by \citet{eichfelderadaptive} and heuristic approaches given by \citet{laumanns2002combining}.
\par An unconstrained uncertain multiobjective optimization problem with a parameter uncertainty (perturbation) set consists of a set of deterministic multiobjective optimization problems, which can be expressed as follows:
\begin{equation}\label{0.2}
	P(U)=\{P(\xi):\xi\in U\},
\end{equation}
where for any fix $\xi\in U$, $$P(\xi):\displaystyle \min_{x\in \mathbb{R}^n}\zeta(x,\xi)=( \zeta_1(x,\xi),\zeta_2(x,\xi),..., \zeta_m(x,\xi)), ~~ \xi\in U\subset \mathbb{R}^k$$
$\zeta:\mathbb{R}^n\times U\to\mathbb{R}^m$ and $U$ is uncertainity set.
\par The concept of robust optimization, originally developed for uncertain single optimization problems by \citet{soyster1973convex,ben2009robust}, has been extended and analyzed for uncertain multiobjective optimization problems \citet{kuroiwa2012robust,ehrgott2014minmax}.  \citet{kuroiwa2012robust},  introduced the concept of minimax robustness for multiobjective optimization problems, replacing the objective vector of the uncertain problem with a vector that incorporates worst-case scenarios for each objective component. Further,  \citet{ehrgott2014minmax} generalized the concept of minimax robustness for multiobjective optimization in a comprehensive manner. This robustness concept, also known as robust optimization, enables the transformation of an uncertain multiobjective optimization problem into a deterministic multiobjective optimization problem, referred to as the robust counterpart. There are several type of robust counterpart used to change uncertain multiobjective optimization problem into deterministic multiobjective optimization problem. By \citet{ehrgott2014minmax}, the minimax approach and the objective-wise worst-case cost type robust counterpart are employed to transform the uncertain multiobjective optimization problem into a deterministic one, which can then be solved using scalarization techniques such as the weighted sum method and the $\epsilon$-constraint method. For further details of the weighted sum method and the $\epsilon$-constraint method see \citet{ehrgott2005multicriteria}.
 \par Scalarization techniques have drawbacks for uncertain multiobjective optimization problems (see \citet{ehrgott2014minmax}). These approaches require the pre-specification of weights, constraints, or function importance, which is not known in advance.
To address these difficulties, \citet{shubham2023newton} proposed solving uncertain multiobjective optimization problems using Newton's method by using objective wise worst case cost type robust counterpart (OWC). Generally, OWC is a non-smooth deterministic multiobjective optimization problem and cannot be handled using smooth multiobjective optimization techniques which is given in the literature \citet{fliege2000steepest,drummond2005steepest,fliege2009newton,ansary2020sequential,ansary2019sequential,fukuda2014survey,ansary2015modified,lai2020q,ansary2021globally,gass1955computational,povalej2014quasi}. \citet{shubham2023newton} extended \citet{fliege2009newton} Newton's method for deterministic multiobjective optimization problems to uncertain multiobjective optimization problems using OWC. However, Newton's method is computationally expensive. To overcome this issue, our paper aims to solve uncertain multiobjective optimization problems using OWC and the quasi Newton method.\\
Quasi Newton algorithms, initially introduced in 1959  by \citet{davidonvariable} and later popularized by \citet{fletcher1963rapidly}, are a class of methods for solving single-objective optimization problems. These algorithms are known for avoiding second-order derivative computations and performing well in practice. In the quasi Newton method, the search direction is determined based on a quadratic model of the objective function, where the true Hessian is replaced by an approximation that is updated after each iteration. Quasi Newton algorithms have also been extended for multiobjective optimization problems \citet{qu2011quasi}. BFGS (Broyden-Fletcher-Goldfarb-Shanno) is the most effective quasi Newton update formula for single-objective optimization problems and has been extended for smooth multiobjective optimization problems. Under suitable assumptions, it has been proven that the quasi Newton algorithm with BFGS update and appropriate line search is globally convergent for strongly convex functions \citet{qu2011quasi}.
\par A Newton method has also been proposed for uncertain multiobjective optimization problems via a robust optimization approach by \citet{shubham2023newton}. In Newton's method, the search direction is computed based on a quadratic model of the objective function, where the Hessian of each objective function over the element of uncertainty set is positive definite at each iteration. However, it is not an easy task to compute the Hessian matrix at every step of the method given in \citet{shubham2023newton}. Therefore, we present a quasi-Newton method for uncertain multiobjective optimization problems, in which a search direction is found by using the approximation of the Hessian matrix of each objective function over the element of uncertainty set at every iteration. The approximation of the Hessian matrix does not make the calculation easier only, but it makes the algorithm faster as well.
\par %To our knowledge, there is currently no existing quasi Newton method for the robust counterpart of uncertain multiobjective optimization problems. Therefore, in this study, we have extended the idea proposed by \citet{povalej2014quasi} to handle the robust counterpart of uncertain multiobjective optimization problems.
In literature, there is currently no existing quasi Newton method for the robust counterpart of uncertain multiobjective optimization problems. \citet{povalej2014quasi} studied the quasi Newton method for multiobjective optimization problems. In our study, we have extended the idea of \citet{povalej2014quasi} for the robust counterpart of uncertain multiobjective optimization problems.
\par In this paper, we employ the quasi Newton method to solve the $P(U)$ under consideration uncertainty set $U$ is finite by using objective-wise worst-case (OWC). The OWC problem is generally a non-smooth multiobjective optimization problem, which means that the BFGS update formula designed for smooth multiobjective optimization problems is not applicable. To address this issue, we modify the BFGS update formula with Armijo type inexact line search technique and incorporate it into the quasi Newton algorithm for the OWC problem. Furthermore, we demonstrate the convergence of the quasi Newton algorithm to the Pareto optimal point of the OWC problem with a superlinear rate. Importantly, the Pareto optimal solution for the OWC problem also serves as the robust Pareto optimal solution for $P(U)$.

\par The paper is structured as follows: Section $\ref{s2}$ provides important results, basic definitions, and theorems relevant to our problem. In subsection $\ref{ss3.1}$, we present the solution to the quasi Newton descent direction finding subproblem. Subsequently, in subsection $\ref{ss3.2}$, we introduce an Armijo-type line search rule to determine a suitable step size, ensuring a decrease in the function value along the quasi Newton descent direction.  Subsection $\ref{ss3.3}$ presents a modified BFGS-type formula and the quasi Newton's algorithm for $RC^*_1$. By utilizing this algorithm, we generate a sequence, and its convergence to a critical point is proven in subsection $\ref{ss3.4}$. In subsection $\ref{ss3.5}$, we establish that under certain assumptions, the sequence generated by the quasi Newton algorithm converges to a Pareto optimum with a superlinear rate. Finally, in subsection $\ref{ss3.6}$, we numerically verify the quasi Newton's algorithm for $RC^*_1$ using appropriate examples. Section $\ref{s4}$ concludes the paper with remarks on the algorithm.

%%%%%%%%%%%%%%%%%%%%%%%%%%%%%%%%%%%%%%%%%%%%%%%%%%%%%%%%%%%%%%%%%%%%%%%%%%%%%%%%%%%

\section{Preliminaries}\label{s2}
Let's introduce some notations.
$\mathbb{R}$: the set of real numbers,
$\mathbb{R}_{\geq}$: the set of non-negative real numbers,
$\mathbb{R}_{>}$: the set of positive real numbers,
$\mathbb{R}_{>} = \{x \in \mathbb{R} : x > 0\}$,
$\mathbb{R}_{\geq} = \{x \in \mathbb{R} : x \geq 0\}$,
$\mathbb{R}^{n} = \mathbb{R} \times ... \times \mathbb{R}$ ($n$ times),
$\mathbb{R}^{n}_{\geq} = \mathbb{R}_{\geq} \times ... \times \mathbb{R}_{\geq}$ ($n$ times),
$\mathbb{R}^{n}_{>} = \mathbb{R}_{>} \times ... \times \mathbb{R}_{>}$ ($n$ times).
In the context of $p, q \in \mathbb{R}^n$, the notation $p \leq q$ means that each component of $p$ is less than or equal to the corresponding component of $q$ ($p_i \leq q_i$ for all $i = 1, 2, ..., n$).
Similarly, for any $p, q \in \mathbb{R}^n$: $p \geq q$ if and only if $p - q \in \mathbb{R}^n_{\geq}$, which is equivalent to $p_i - q_i \geq 0$ for each $i$,
$p > q$ if and only if $p - q \in \mathbb{R}^n_{>}$, which is equivalent to $p_i - q_i > 0$ for each $i$.
Lastly, we denote the indexed sets as $\bar\Lambda = \{1, 2, ..., p\}$ and $\Lambda = \{1, 2, ..., m\}$, which contain $p$ and $m$ elements, respectively.

\par Now we define Pareto optimal solution (efficient solution) for problem $P$ defined in $(\ref{0.1})$. A feasible point $x^*\in D$ is said to be an efficient (weakly efficient) solution for $P$ if there is no other $x\in D$ such that $\zeta(x)\leq\zeta(x^*)$ $\&$  $\zeta(x)\not =\zeta(x^*)$ $\big(\zeta(x)<\zeta(x^*)\big).$ If $x^*$ is an efficient (weakly efficient) solution, then $\zeta(x^*)$ is called non dominated  (weakly non dominated) point, and the set of efficient solution and non dominated point are called efficient set and non dominated set respectively. A feasible point $x^*$ is said to be locally efficient (locally weakly efficient) solution if $D$ is restricted to a neighborhood of $x^*$ as a efficient (weakly efficient) solution i.e., if there is a neighborhood $N\subset D$ of $x^*$ such that $x^*$ is an efficient solution (weakly efficient solution) on $N$. In case, if $\zeta(x)$ is $\mathbb{R}^m-$convex (i.e., each component of $\zeta(x)$ is convex) then locally efficient solution will be globally efficient solution.   \citet{gopfert1990vektoroptimierung}, \citet{LucDT(1988)} given the necessary condition for locally weakly efficient solution for smooth deterministic multiobjective optimization problem (\ref{0.1}). The necessary condition for a given point $x^*$ to be a locally weakly efficient solution for (\ref{0.1}) is $I(J(\zeta(x^*)))\cap(-\mathbb{R}^m_{>})=\emptyset,$ where $I(J(\zeta(x^*))$ denotes the image set of Jacobian of $\zeta$ at $x^*$.  \citet{fliege2000steepest} given the definition of critical point for smooth deterministic multiobjective optimization problem (\ref{0.1}). A point $x^*$ is said to be locally critical point for smooth deterministic multiobjective optimization problem (\ref{0.1}) if $I(J(\zeta(x^*)))\cap(-\mathbb{R}^m_{>})=\emptyset$. From the definition of critical point it is clear that if $x^*$ is not a critical point then $\exists$ a $d\in\mathbb{R}^n$ such that $\nabla \zeta_j(x^*)^Td<0$ for all $j\in \Lambda.$
\par In this paper, our focus is on finding the solution to $P(U)$ (as defined in equation (\ref{0.2})), where the uncertainty set $U$ is a finite nonempty set. To achieve this, we transform the problem into a deterministic multiobjective optimization problem using the worst-case cost type robust counterpart. We then solve this transformed problem using the quasi Newton method. In this case, we assume that $U$ is a finite uncertainty set with $p$ elements. Thus, the uncertain multiobjective optimization problem $P(U)$ with a finite uncertainty set $U$ can be reformulated as follows:
\begin{equation}\label{mp}
P(U)=\{P(\xi_i):i\in\bar{\Lambda}\}, ~U=\{\xi_i:i\in\bar{\Lambda}\},
\end{equation}
where for any fix $\xi_i\in U$, $$P(\xi_i):\displaystyle \min_{x\in \mathbb{R}^n}\zeta(x,\xi_i)=( \zeta_1(x,\xi_i),\zeta_2(x,\xi_i),..., \zeta_m(x,\xi_i)), ~~ \xi_i\in U\subset \mathbb{R}^k,$$
and its objective wise worst case cost type robust counterpart($OWC$), $RC^*_1$ is given as follows:
\begin{equation}\label{1.2}
	RC^*_1:~~\min_{x\in \mathbb{R}^n}F(x),~ \text{where}~F(x)=(F_1(x),F_2(x),\ldots,F_m(x)),~F_j(x)=\max_{i\in\bar{\Lambda}}\zeta_j(x,\xi_i).
\end{equation}
Additionally, let's define $I_j(x)$ as the set of all points in $\bar{\Lambda}$ where $\zeta_j(x, \xi_i)$ achieves its maximum value. In other words, $I_j(x)$ is given by $I_j(x) = \{i \in \bar{\Lambda} : \zeta_j(x, \xi_i) = F_j(x)\}$. Further details can be found in \citet{sun2006optimization}.

For the problem $P(U)$, our objective is to find a robust efficient solution or a robust weakly efficient solution, which can be defined as follows:
\begin{dfn}\citet{ehrgott2014minmax}
	A feasible point $x^*\in\mathbb{R}^n$ is said to be robust efficient (robust weakly efficient) if there is no  $x\in\mathbb{R}^n\setminus\{x^*\}$ such that $\zeta(x;U)\subseteq \zeta(x^*;U)-\mathbb{R}^k_\ge$ \big($\zeta(x;U)\subseteq \zeta(x^*;U)-\mathbb{R}^k_>$\big), where $\zeta(x;U)=\{\zeta(x,\xi):\xi \in U\}$ set of all possible values of objective function.
	\end{dfn}
Note that $RC^*_1$ represents a deterministic non-smooth multiobjective optimization problem. Solving this problem will yield a solution in the form of Pareto optimal or weak Pareto optimal solutions. To establish the connection between $RC^*_1$ and $P(U)$, M. Ehrgott et al. presented a theorem that demonstrates that the solution of $RC^*_1$ coincides with the solution of $P(U)$.
\begin{thm}\label{t1}\rm Let $P(U)$ be an uncertain multiobjective optimization problem with finite uncertainty set which is defined in (\ref{mp})  and $RC^*_1$ be the robust counterpart of $P(U)$. Then:
	\begin{enumerate}[(a)]\rm
		\item If $x^*\in \mathbb{R}^n$ is a efficient solution to $RC^*_1$. Then $x^*$ is robust efficient solution for $P(U)$.
		\item If $\displaystyle\max_{i\in \bar{\Lambda}} \zeta_j(x,\xi_i)$  exist for all $j\in\Lambda$ and all $x \in \mathbb{R}^n$ and $x^*$
		is a weakly efficient solution to $RC^*_1.$ Then $x^*$ is robust weakly efficient solution for $P(U)$.
	\end{enumerate}
\end{thm}
\begin{proof}
\text{(a)} We can prove the result by contradiction. On contary, we assume that $x^*$ is not a robust efficient solution for $P(U)$ then by definition of robust efficient solution there exists $\bar{x}\in \mathbb{R}^n$ such that
	\begin{align*}&\zeta(\bar{x};U)\subseteq \zeta(x^*;U)-\mathbb{R}^k_\ge,\\  &\Longrightarrow  \zeta(\bar{x};U)-\mathbb{R}^k_\ge\subseteq \zeta(x^*;U)-\mathbb{R}^k_\ge\\
	&	\Longrightarrow ~\text{ for all }~ \xi_i\in U ~\text{there exist a }~\xi_l\in U~ \text{such that}\\
	&	\Longrightarrow  \zeta(\bar{x},\xi_i)\leq \zeta(x^*,\xi_l)\\
	&	\Longrightarrow  \displaystyle \max_{\xi_i\in U}\zeta_j(\bar{x},\xi_i)\leq \displaystyle \max_{\xi_l\in U}\zeta_j(x^*,\xi_l)~\text{for all} ~j\in \Lambda.
	\end{align*}
By the last expression, it is clear that $F(x^*)$ is equal to or dominated by $F(\bar{x})$ in $RC^*_1,$ thus $x^*$ is not an efficient solution for $RC^*_1,$ which contradicts the fact that $x^*$ is an efficient solution for $RC^*_1.$ Hence, $x^*$ is a robust efficient solution for $P(U).$\\
\text{(b)}  Similar to item (a), item (b) can be prove by the contradiction. On contrary, we assume that $x^*$ is not a robust weak efficient solution for $P(U)$ then by definition of robust efficient solution there exists $\bar{x}\in \mathbb{R}^n$ such that
\begin{align*}&\zeta(\bar{x};U)\subseteq \zeta(x^*;U)-\mathbb{R}^k_>,\\  &\Longrightarrow  \zeta(\bar{x};U)-\mathbb{R}^k_\ge\subseteq \zeta(x^*;U)-\mathbb{R}^k_>\\
	&	\Longrightarrow ~\text{ for all }~ \xi_i\in U ~\text{there exist a }~\xi_l\in U~ \text{such that}\\
	&	\Longrightarrow  \zeta(\bar{x},\xi_i)< \zeta(x^*,\xi_l)\\
	&	\Longrightarrow  \displaystyle \max_{\xi_i\in U}\zeta_j(\bar{x},\xi_i)< \displaystyle \max_{\xi_l\in U}\zeta_j(x^*,\xi_l)~\text{for all} ~j\in \Lambda.
\end{align*}
By the last expression, it is clear that $F(x^*)$ is strictly dominated by $F(\bar{x})$ in $RC^*_1,$ thus $x^*$ is not a weak efficient solution for $RC^*_1,$ which contradicts the fact that $x^*$ is a weak efficient solution for $RC^*_1.$ Hence, $x^*$ is a robust weak efficient solution for $P(U).$
\end{proof}
According to Theorem $\ref{t1}$, we can observe that solution of $RC^*_1$ is the solution $P(U)$. Therefore, we employ the quasi Newton algorithm to solve $RC^*_1$. To do so, we first establish the necessary condition for efficiency or Pareto optimality. For a given point $x^* \in \mathbb{R}^n$ to be a locally weakly efficient solution of $RC^*_1$, it is required that $I(Conv\{\displaystyle \cup_{j\in \Lambda} \partial F_j(x^*)\})\cap(-\mathbb{R}^{n}_{>})= \emptyset.$ Motivated by this notion, we introduce the concept of a critical point for $RC^*_1$.
%\begin{thm}\citet{my paper}
\begin{dfn}\label{d2.2}\citet{shubham2023newton} A point $x^*\in\mathbb{R}^n$ is said to be critical point for $RC^*_1$ if  $$I(Conv\{\displaystyle \cup_{j\in \Lambda} \partial F_j(x^*)\})\cap(-\mathbb{R}^{n}_{>})= \emptyset.$$
	\end{dfn}
Note that if $x^*$ is a critical point for $RC^*_1$ then there is no $v\in \mathbb{R}^n$ such that $\nabla\zeta_j(x^*,\xi_i)^Tv<0$, for all $i\in I_j(x),$ $j\in \Lambda$, where $I_j(x)=\{i\in\bar{\Lambda}:\zeta_{j}(x^*,\xi_i)=F_j(x^*)\}.$
\begin{dfn}\citet{nocedal2006wright}
		Let $ f:\mathbb{R}^n \to \mathbb{R}$ be a function. Then a vector $v$ is said to be descent direction of $f$ at $x$ if and only if there exists $\delta > 0$ such that $f(x + \alpha v) < f(x),$ $\forall\alpha\in (0, \delta).$ If $f$ is continuously differentiable then a vector $v\in\mathbb{R}^n$ is said to be a descent direction for $f$ at $x$ if $\nabla f(x)^Tv< 0$.
\end{dfn}
\begin{dfn}\citet{shubham2023newton}
	 In case of $RC^*_1$, a vector $v$ is said to be descent direction for $F$ at $x$ if $$\nabla \zeta_j(x,\xi_i)^Tv< 0, ~\forall j \in \Lambda$$ and $i\in I_j(x).$
Also if $v$ is descent direction for $F(x)$ at $x$ then there exists $\epsilon >0$ such that  $$F_j(x+\alpha v)< F_j(x) ~\forall~ j \in \Lambda~ \text{and} ~\alpha \in(0,\epsilon].$$
\end{dfn}
\begin{thm}\label{directional derivative}\citet{dhara2011optimality} Let $F_j: \mathbb{R}^n \to  \mathbb{R}$ be a function such that $F_j(x)=\max_{i\in\bar{\Lambda}}\zeta_j(x,\xi_i).$ Then:
	\begin{enumerate}[(i)]
		\item The directional derivative of $F_j$ at $x$ in the direction $v$ is given by
		$F'_j(x,v)=\displaystyle\max_{i\in I_j(x)}\nabla\zeta_j(x,\xi_i)^Tv,$ where $I_j(x) = \{i \in \bar{\Lambda} : \zeta_j(x, \xi_i) = F_j(x)\}$.
		\item\label{ddii} The subdifferential of $F_j$ is  \[\partial F_j(x)= Conv \bigg(\bigcup_{i\in I_j(x)}\zeta_j(x,\xi_i)\bigg).\] Also $x^*=\displaystyle arg\min_{x\in \mathbb{R}^n}F_j(x)$ if and only if $0\in\partial F_j(x^*)).$
		\end{enumerate}
	\end{thm}
%\begin{dfn}\label{1**}\citet{15}
\begin{dfn}\label{1**}\citet{miettinen1999nonlinear} Let $f: S\subset\mathbb{R}^n \to  \mathbb{R}^m$ be a continuously differentiable function such that $f(x)=(f_1(x),f_2(x),\ldots,f_m(x))$ then
	$f$ is $\mathbb{R}^m-\text{strongly convex}$ if and only if there exists $\omega>0$ such that $\omega\|u\|^2 \leq u^T\nabla^2f_j(x)u, \forall x \in S, u\in\mathbb{R}^n ~\text{and} ~j \in \Lambda.$
\end{dfn}
\par To prove the necessary and sufficient condition for Pareto optimality for $RC^*_1$  we give lemma \ref{lm1}.
\begin{lem}\label{lm1}
	If $x^*$ is critical point for $F$ if and only if $0\in Conv\{\displaystyle \bigcup_{j\in \Lambda}\partial F_j(x^*)\}$.
\end{lem}
\begin{proof}
	Since $x^*$ is critical point for $F$ then there must exists $d\in\bigcup_{j\in \Lambda}\partial F_j(x^*)$ such that
	\begin{equation}\label{ccc}
		v^Td\geq0,~\forall v\in\mathbb{R}^n.
	\end{equation} On contrary, if we assume $0\not\in Conv\{\displaystyle \bigcup_{j\in \Lambda}\partial F_j(x^*)\}.$ Since, $Conv\{\displaystyle \bigcup_{j\in \Lambda}\partial F_j(x^*)\}$ and $\{0\}$ are closed and convex sets  then with the help of separation theorem, there exists $v\in\mathbb{R}^n$ and $b\in \mathbb{R}$ such that
	$v^T0\geq b$ and $v^Td<b~\forall d\in Conv\{\displaystyle \bigcup_{j\in \Lambda}\partial F_j(x^*)\}.$ Jointly both inequality contradicts (\ref{ccc}). Hence,  $0\in Conv\{\displaystyle \bigcup_{j\in \Lambda}\partial F_j(x^*)\}$. Conversely, we have to prove that if $0\in Conv\{\displaystyle \bigcup_{j\in \Lambda}\partial F_j(x^*)\}$ then $x^*$ is critical point for $F$. For that we define $F'(x)=\displaystyle\max_{j\in\Lambda} F_j(x)- F_j(x^*).$ Then by item (\ref{ddii}) of Theorem \ref{directional derivative} we have  $\partial F'(x)= Conv \bigg(\displaystyle\bigcup_{i\in\Lambda}\partial\zeta_j(x)\bigg).$ Therefore, by assumption  we get $0\in Conv \bigg(\displaystyle\bigcup_{i\in\Lambda}\partial\zeta_j(x)\bigg)$ which implies $x^*=\arg\displaystyle \min_{x\in\mathbb{R}^n}F'(x).$ On contrary, if we assume $x^*$ is not a critical point then by Definition \ref{d2.2}, there exists $s\in \mathbb{R}^n$ such that $\nabla\zeta_j(x^*,\xi_i)^Ts<0$, for all $i\in I_j(x^*),$ $j\in \Lambda$ i.e., $F'_j(x^*,s)<0$ for all $j.$ Then there exists some $\eta>0$ sufficiently small such that $F_j(x^*+\eta s)<F_j(x^*)$ for all $j$ which implies $F'(x^*+\eta s)<0=F'(x^*)$ holds for some $x^*+\eta s\in \mathbb{R}^n$. Then there is a contradiction to the fact that $x^*=\arg\displaystyle \min_{x\in\mathbb{R}^n}F'(x).$ Therefore our assumption $x^*$ is not a critical point is wrong and hence $x^*$ is a critical point for $F.$
\end{proof}
We now prove the necessary and sufficient condition for the Pareto optimality for the robust counterpart $RC^*_1$ using Lemma \ref{lm1}.
\begin{thm}\label{nce}
	If  $\zeta_j(x,\xi)$ is continuously differentiable and convex for each $j\in \Lambda$ and $\xi\in U$, then  $x^*$ is a weakly efficient solution for $RC^*_1$ if and only if
$$	0 \in conv \left( \displaystyle \cup_{j=1}^{m}\partial F_j(x^*) \right).$$
\end{thm}
\begin{proof}
	Let  $x^*$ be a weakly efficient solution for  $RC^*_1$. We have to show $0\in Conv\{\displaystyle \cup_{j\in \Lambda}\partial F_j(x^*)\}$.
	Since given function $\zeta_j(x,\xi_i)$ is continuously differentiable and convex for each $j$ and $\xi_i \in U$, then $\zeta_j(x,\xi_i)$ will be locally Lipschitz continuous for all $i\in \bar{\Lambda}$. Then, by Theorem 4.3 in \citet{makela2014nonsmooth} $0\in Conv\{\displaystyle \cup_{j\in \Lambda}\partial F_j(x^*)\}$. Conversely, by assumption $0\in Conv\{\displaystyle \cup_{j\in \Lambda}\partial F_j(x^*)\}$ it is clear that $x^*$ is critical point. Then, for atleast one $j^0$, we have $F'_{j^0}(x^*,d) \geq 0, ~\forall~ d\in \mathbb{R}^n$. Now, by using the Definition \ref{directional derivative}, we have
	\begin{equation}\label{Equ1}
		\nabla \zeta_{j^0}(x^*,\xi_i)^Td\geq 0, ~\forall~ d\in \mathbb{R}^n,~ i\in I_{j^0}(x^*).
	\end{equation}
	By convexity of $F_j$ and $\zeta_j(x,\xi_i)$, we get
	\begin{center}
		$\zeta_{j^0}(x,\xi_i)\geq \zeta_{j^0}(x^*,\xi_i) + \nabla \zeta_{j^0}(x^*,\xi_i)^T(x-x^*),~ \forall ~i\in I_{j^0}(x^*)$ and $x, ~x^* \in \mathbb{R}^n.$
	\end{center}
	By (\ref{Equ1}), we have
	\begin{center}
		$ \zeta_{j^0}(x,\xi_i)\geq \zeta_{j^0}(x^*,\xi_i),~ \forall~ i\in I_{j^0}(x^*),$
	\end{center}
	and therefore
	\begin{center}
		$ F_{j^0}(x)\geq F_{j^0}(x^*),~ \forall x\in \mathbb{R}^n,$
	\end{center}
	i.e., $x^*$ is weakly efficient solution.
	% Thus, this theorem gives us necessary   condition for Pareto optimal for $\phi$.
	%Since the function  $f_j(x,\xi)$ is continuously differentiable and convex for each $j\in \Lambda$ and $\xi \in U$, the necessary condition (Theorem \ref {Thm 2.1}) will be sufficient. (Thus, this theorem gives us a necessary  sufficient condition for Pareto optimality for (RDMOP)).
\end{proof}
%	\begin{lem}\label{lm1}\rm\citet{shubham2023newton}
%	If $x^*$ is critical point for $F$ then $0\in Conv\{\displaystyle \cup_{j\in \Lambda}\partial F_j(x^*)\}$.
%\end{lem}
%\begin{thm}\label{t2.2}\rm \citet{shubham2023newton}
%	Let  $\zeta_j(x,\xi_i)$ be continuously differentiable and convex function for each $j$ and $\xi_i \in U$ and if $x^*$ is weakly efficient solution for $RC^*_1$ then $0\in Conv\{\displaystyle \cup_{j\in \Lambda}\partial F_j(x^*)\}$.
%\end{thm}
%\begin{thm}\label{t2.2}\rm \citet{shubham2023newton}
%	Suppose $F$ is $\mathbb{R}^m$-convex function and if $x^*$ is critical point for $RC^*_1$  then $x^*$ is weak efficient solution of $RC^*_1$.
%\end{thm}
\par BFGS (Broyden, Fletcher, Goldfarb and Shanno) is a most popular quasi Newton method for single objective optimization problem. It is a line search method with a descent direction $$ v^k=-(B^k)^{-1}\nabla f(x^k)$$ where $f:S\to \mathbb{R}$ is a continuously differentiable function, $S\subset \mathbb{R}^n$ and $B^k$ is a positive definite symmetric matrix which is updated at every iteration. Also the new iteration is given by $$x^{k+1}=x^k+\alpha_{k}v^k$$ where $v^k$ and $\alpha_{k}$ are the descent direction and steplenth. Now the BFGS update formula is given as follows
\begin{equation}\label{*}
B^{k+1}=B^k-\frac{B^ku^k(u^k)^TB^k}{(u^k)^TB^ku^k}+\frac{p^k(p^k)^T}{(u^k)^T p^k}
\end{equation}
where $u^k= x^{k+1}-x^k=\alpha_kv^k$ and $p^k=\nabla f(x^{k+1})-\nabla f(x^{k}).$ Also from \citet{nocedal2006wright}, $B^{k+1}$ is positive definite whenever $B^k$ is positive definite and $\alpha_k$ satisfies Wolfe line search rule. For the BFGS method two conditions must be satisfied. The seacant equation $B^{k+1}u^k=p^k$ is the first condition which we obtain by multiplying $(\ref{*})$ by $u^k$. The curvature condition $(u^k)^T p^k>0$ is the second necessary condition. Also if $f$ is strongly convex, then for any two points $x^k$ and $x^{k+1}$ the curvature condition is satisfied.
BFGS update formula which is given in (\ref{*}) defined for single objective optimization problem, and \citet{povalej2014quasi}, generalized this idea for unconstrained multiobjective optimization problem of strongly convex objective functions. The BFGS update formula, given by {\v{Z}}iga Povalej is as follows
\begin{equation}\label{bfgsformop}
	B_j^{k+1}=B_j^k-\frac{B_j^ku^k(u^k)^TB_j^k}{(u^k)^TB_j^ku^k}+\frac{p^k(p^k)^T}{(u^k)^T p^k}
\end{equation}
where $u^k= x^{k+1}-x^k=\alpha_kv^k$ and $p^k=\nabla f_j(x^{k+1})-\nabla f_j(x^{k}).$
\par In the subsequent section, we develop a quasi-Newton method for $RC^*_1$ defined in \ref{1.2}. To write the quasi-Newton algorithm, search direction,  stepsize in the search direction, and Hessian approximation update formula are the three main required things.
%In the subsequent subsection, we formulate a subproblem and solve it in order to determine the descent direction for $RC^*_1$.
\section{Quasi Newton method for \texorpdfstring{$RC^*_1$}{Lg}}
In this section, we establish the necessary results to formulate the quasi Newton algorithm. The algorithm requires three main components: (1) the quasi Newton descent direction, (2) an appropriate step length along the descent direction, and (3) an updated formula to maintain the positive definiteness of the Hessian approximation at the current iteration point.

To obtain these components, we consider $RC^*_1$ with a function $F: \mathbb{R}^n \rightarrow \mathbb{R}^m$ defined as $F(x) = (F_1(x), F_2(x), ..., F_m(x))$, where $F_j(x)= \displaystyle \max_{i\in \bar{\Lambda}} \zeta_j(x,\xi_i),~ j\in \Lambda$. Here, $\zeta_j:\mathbb{R}^n \times U \to \mathbb{R}$ represents a twice continuously differentiable and strictly convex function for each $x$ and $\xi_i \in U$.

Next, we construct a subproblem that will provide the quasi Newton descent direction once solved.

\subsection{Subproblem to find the quasi Newton descent direction for \texorpdfstring{$RC^*_1$}{Lg}}
\begin{equation}\label{ss3.1}
	\displaystyle \min_{s\in \mathbb{R}^n}\Theta_{x}(s),
\end{equation}
where
$\Theta_{x}(s)= \displaystyle\max_{j\in \Lambda} \max_{i\in \bar{\Lambda}}\{ \zeta_j(x,\xi_i)+ \nabla \zeta_j(x,\xi_i)^Ts +\frac{1}{2}s^TH_j(x,\xi_i)s-F_j(x)\}.$\\
In the subproblem (\ref{ss3.1}), $H_j(x,\xi_i)$ is the approximation of the Hessian $\nabla^2\zeta_j(x,\xi_i),$ for each $j\in\Lambda$, $i\in\bar{\Lambda}.$ Also, we are assuming that $H_j(x,\xi_i)$ is positive definite for each $j\in\Lambda$, $i\in\bar{\Lambda}.$ Then the objective function $\Theta_{x}(s)$ is strongly convex being a maximum of strongly convex function and hence subproblem (\ref{ss3.1}) is well defined and has a unique solution. If we denote optimal solution and optimal value of the subproblem (\ref{ss3.1}) by $s(x)$ and $\vartheta(x)$ respectively, then we can write
\begin{equation}
	s(x)=\displaystyle arg\min_{s\in \mathbb{R}^n}\Theta_{x}(s)
	\end{equation}
\begin{equation}
	\vartheta(x)=\Theta_{x}(s(x)).
\end{equation}
\begin{thm}\rm
	Let at any fix $x\in\mathbb{R}^n$, $H_j(x,\xi_i)$ be a positive definite  approximation of the Hessian $\nabla^2\zeta_j(x,\xi_i),$ for each $j\in\Lambda$, $i\in\bar{\Lambda},$ then the subproblem (\ref{ss3.1}) has a unique solution, namely $s(x)$ which is given by
	\begin{equation}\label{dd}
		s(x) = -\bigg(\sum\limits_{j\in \Lambda}  \sum\limits_{i\in\bar{\Lambda}}\lambda_{ij} H_j(x,\xi_i)\bigg) ^{-1}	\sum\limits_{j\in \Lambda}  \sum\limits_{i\in \bar{\Lambda}} \lambda_{ij} \nabla \zeta_j(x,\xi_i).
	\end{equation}
\end{thm}
\begin{proof}
We want to find the solution of subproblem (\ref{ss3.1}) which is an unconstrained real-valued minimization problem.  Equivalently, by replacing $\Theta_{x}(s)$ as $t\in\mathbb{R}$ the subproblem (\ref{ss3.1}) can be written as
	\begin{center}
		$Q_N(x):$ $\displaystyle \min_{(t,s)\in \mathbb{R}\times\mathbb{R}^n}t$
		
		$ s.t. ~~~~~~~ \zeta_j(x,\xi_i)+ \nabla \zeta_j(x,\xi_i)^Ts +\frac{1}{2}s^TH_j(x,\xi_i)s-F_j(x)\leq t, ~~~\forall i\in \bar{\Lambda} ~\text{and} ~j\in \Lambda.$
		\end{center}
	Since the problem $Q_N(x)$ is a convex programming problem then the unique solution of  $Q_N(x)$ is given by $(t,s)=(\vartheta(x),s(x))$. Also $Q_N(x)$ has a Slater point then solution of this problem is given by the KKT optimality condition. For that the Lagrangian function is given by
	\begin{align*}
		L(t,s,\lambda) &= t + \sum\limits_{j\in \Lambda}  \sum\limits_{i\in \bar{\Lambda}} \lambda_{ij} ( \zeta_j(x,\xi_i)+ \nabla \zeta_j(x,\xi_i)^Ts +\frac{1}{2}s^TH_j(x,\xi_i)s-F_j(x)-t).
	\end{align*}
	Then the KKT optimality conditions are
	\begin{equation} \label{2}
		\sum\limits_{j\in \Lambda}  \sum\limits_{i\in \bar{\Lambda}} \lambda_{ij}=1
	\end{equation}
	\begin{equation}\label{3}
		\sum\limits_{j\in \Lambda}  \sum\limits_{i\in \bar{\Lambda}} \lambda_{ij} \nabla \zeta_j(x,\xi_i)+ \sum\limits_{j\in \Lambda}  \sum\limits_{i\in \bar{\Lambda}}\lambda_{ij} H_j(x,\xi_i)^Ts=0
	\end{equation}
	\begin{equation}
		\ \zeta_j(x,\xi_i)+ \nabla \zeta_j(x,\xi_i)^Ts +\frac{1}{2}s^TH_j(x,\xi_i)s-F_j(x)-t\leq 0,~~~\forall i\in\bar{\Lambda} ~\text{and}~ j\in \Lambda
	\end{equation}
	\begin{equation}
		\lambda_{ij}(\zeta_j(x,\xi_i)+ \nabla \zeta_j(x,\xi_i)^Tv +\frac{1}{2}s^TH_j(x,\xi_i)s-F_j(x)-t)= 0, ~~\lambda_{ij}\geq 0.
	\end{equation}
	From $(\ref{3})$, ~~~~~~~~~~~
	\begin{equation}\label{dd}
		s(x) = -\bigg(\sum\limits_{j\in \Lambda}  \sum\limits_{i\in\bar{\Lambda}}\lambda_{ij} H_j(x,\xi_i)\bigg) ^{-1}	\sum\limits_{j\in \Lambda}  \sum\limits_{i\in \bar{\Lambda}} \lambda_{ij} \nabla \zeta_j(x,\xi_i).
	\end{equation}
	Since $ H_j(x,\xi_i)$ is positive definite for all $x\in\mathbb{R}^n,$  $i\in\bar{\Lambda},$ and $j\in\Lambda$, and $\lambda_{ij}\geq0$ and $x$ is not a critical point, it follows from equation (\ref{2}) that there must exists at least one $\lambda_{ij}>0$. Consequently, the inverse of $\sum\limits_{j\in \Lambda}  \sum\limits_{i\in \bar{\Lambda}}\lambda_{ij} H_j(x,\xi_i)$ exists. Therefore $s(x)$ is well defined. Moreover, the optimal value of the subproblem can be expressed as follows:
	\begin{equation}\label{ov}
		\vartheta(x)=\max_{j\in \Lambda} \max_{i\in \bar{\Lambda}}\{ \zeta_j(x,\xi_i)+ \nabla \zeta_j(x,\xi_i)^Ts(x) +\frac{1}{2}s(x)^TH_j(x,\xi_i)s(x)-F_j(x)\}.
	\end{equation}
\end{proof}
\begin{thm}\label{Thm 3.2}\rm Let $s(x)$ and $\vartheta(x)$ be the solution and optimal value of the subproblem $Q_N(x)$ respectively. Then the following results will hold:
	\begin{enumerate}
		\item\label{th0} $ s(x)$ is bounded on compact subset   $W$ of  $\mathbb{R}^n$ and $\vartheta(x)\leq0$.
		\item The following conditions are equivalent :
		\begin{enumerate}
			\item The point $x$ is not a critical point.
			\item\label{theta<0} $\vartheta(x)<0.$
			\item $s(x)\not =0.$
			\item $s(x)$ is a descent direction for $F$ at $x$ for the problem $RC^*_1$ .
		\end{enumerate}
	\end{enumerate}
	In particular $x$ is critical point iff $\vartheta(x)=0.$
\end{thm}
\begin{proof}$(1)$ Let  $W$ be a compact subset of $\mathbb{R}^n$.	
	Since for each $j\in\Lambda,$ $\zeta_j(x,\xi_i)$ is  twice continuously differentiable for all $x \in \mathbb{R}^n$ and $\xi_i\in U$, then its first and second order derivative will be bounded on compact set  $W$. Also, $H_j(x,\xi)$ is the Hessian approximation of true Hessian of $\zeta_j(x,\xi_i)$ and $H_j(x,\xi)$ positive definite for each $i\in \bar{\Lambda}$ and $j\in \Lambda$. Hence, $s(x)$ is bounded on compact set   $W$.\\
	Since $t=1$ and $s=\bar0=(0,0,\ldots,0)\in \mathbb{R}^n$ lies in the feasible region then we have\\ 	
	$~~~~~~~\vartheta(x) \leq \displaystyle \max_{j\in \Lambda}\{\max_{i\in \bar\Lambda}\zeta_j(x,\xi_i) + \max_{i\in \bar\Lambda}\nabla \zeta_j(x,\xi_i)^T\bar 0 +\max_{i\in \bar\Lambda} \frac{1}{2}\bar 0^TH_j(x,\xi_i)\bar 0-\max_{i\in \bar\Lambda}F_j(x)\}=0.$\\
	Hence $\vartheta(x)\leq0$.\\
	$(2):(a)\implies(b)$\\
	Assume $x$ is not a critical point. Then there exists $\bar s\in\mathbb{R}^n$ such that
	\begin{center}
		$\nabla \zeta_j(x,\xi_i)^T \bar s < 0,  ~\forall j \in \Lambda$ and $i \in I_j(x)$.
	\end{center}
	Since $\vartheta(x)$ is the optimal value for the subproblem (\ref{ss3.1}), then  $\forall\delta >0$ we have
	\begin{center}
		$\vartheta(x) \leq \displaystyle \max_{j\in \Lambda}\{\max_{i\in \bar\Lambda}\zeta_j(x,\xi_i) + \max_{i\in \bar\Lambda}\{\nabla \zeta_j(x,\xi_i)^T(\delta\bar s) + \frac{1}{2}\bar (\delta\bar s)^TH_j(x,\xi_i)(\delta\bar s)\}-\max_{i\in \bar\Lambda}F_j(x)\}\}.$
	\end{center}
	Then we have
	\begin{center}
		$\vartheta(x) \leq \delta\bigg(\ \displaystyle\max_{j\in \Lambda} \max_{i\in \bar{\Lambda}} \{ \nabla\zeta_j(x,\xi_i)^T\bar s +\delta \displaystyle\frac{1}{2}\bar s^TH_j(x,\xi_i)\bar s\}\bigg) $.
	\end{center}
	For small enough $\delta>0$ right hand side of above inequality will be negative because of $\nabla \zeta_j(x,\xi_i)^T\bar s < 0$, $\frac{1}{2} \bar s^TH_j(x,\xi_i)\bar s > 0$ and also $\vartheta(x)\leq 0$.  Thus $\vartheta(x)
	< 0$.\\
	$(b) \implies (c)$\\ Since $\vartheta(x)$ is the optimal value of the subproblem $(\ref{ss3.1})$ and from (b) it is strictly negative so  we will get $s(x) \not =0$. Because if $s(x)  = 0$ then $\vartheta(x)$ will be zero which is not possible from (b). Hence if $\vartheta(x) < 0$ then $s(x) \not = 0$.\\
	$(c) \implies (d)$ \\ Let $s(x) \not = 0$. Then, $\vartheta (x) \not =0$. Since $\vartheta(x) \leq 0$ and $s(x) \not = 0$ then $\vartheta (x)<0$. Thus,
	\begin{align*}
		&\theta(x)\leq\displaystyle \max_{j\in \Lambda} \{\max_{i\in \bar{\Lambda}} \zeta_{j}(x,\xi_i)+ \max_{i\in \bar{\Lambda}} \{\nabla \zeta_j(x,\xi_i)^T s(x) +\frac{1}{2} s^T(x)H_j(x,\xi_i)s(x)\}-\max_{i\in \bar{\Lambda}} F_j(x)\}< 0\\
		&\implies\displaystyle \max_{j\in \Lambda} \max_{i\in \bar{\Lambda}} \{\nabla \zeta_j(x,\xi_i)^T s(x) +\frac{1}{2} s^T(x)H_j(x,\xi_i)s(x)\}< 0\\
		&\implies  \nabla\zeta_j(x,\xi_i)^Ts(x)< 0, ~\forall ~ j\in \Lambda ~\text{and}~ i\in \bar{\Lambda}\\
		& \implies \nabla \zeta_j(x,\xi_i)^Tv(x)< 0, ~\forall ~ j\in \Lambda ~\text{and}~ i\in I_j(x)\\
		&\implies s(x)~ \text{is a descent direction for}~ F~ \text{at} ~x~\text{for problem $RC^*_1$}.
	\end{align*}
	$(d) \implies (a)$  \\ Since $s(x)$ is a descent direction for $F$ at $x$, then we have
	\begin{equation*}
		\nabla \zeta_j(x,\xi_i)^Ts(x)< 0, ~\forall ~ j\in \Lambda ~\text{and}~ i\in I_j(x).
		~\text{Hence}~ x~ \text{is not a critical point.}
	\end{equation*}
	Also  if $\vartheta(x)<0$ then $s(x) \not= 0$ and also if $s(x) \not= 0$ then $\vartheta(x)<0$.  Thus, we have $x$ is  critical point if and only if $\vartheta(x) =0$.
\end{proof}
\par From the Theorem \ref{Thm 3.2}, it is clear that $s(x)\approx 0$ or $\vartheta(x)\approx0$ can be used as a stopping criterion in the quasi Newton algorithm.
Now in the next theorem we will prove $\vartheta(x)$ is continuous for every $x\in \mathbb{R}^n$.
\begin{thm}\rm
	Let $\vartheta:\mathbb{R}^n\to \mathbb{R}$ be a function which is defined in $(\ref{ov})$. Then, $\vartheta(x)$ is continuous.
\end{thm}
\begin{proof}
	To prove continuity of $\vartheta(x)$, it is suffice to prove that $\vartheta(x)$ is continuous in any arbitrary compact subset    $W$ of $\mathbb{R}^n.$ Since, by Theorem \ref{Thm 3.2},  $\vartheta(x)\leq 0$ then for every $j\in\Lambda$ and $i\in\bar{\Lambda}$ we have
	\begin{align*}
	\zeta_j(x,\xi_i)+\nabla \zeta_j(x,\xi_i)^Ts(x) +\frac{1}{2}s(x)^TH_j(x,\xi_i)s(x)-F_{j}(x)\leq0.
	\end{align*}
	Since, at the active indices $\zeta_j(x,\xi_i)$ attains its maximum value i.e., $\zeta_j(x,\xi_i)=F_j(x)$ then from the above equation we get	
	$$\nabla \zeta_j(x,\xi_i)^Ts(x) +\frac{1}{2}s(x)^TH_j(x,\xi_i)s(x)\leq0,$$
	which implies
	\begin{equation}\label{t^1}
		\frac{1}{2}s(x)^TH_j(x,
		\xi_i)s(x)\leq-\nabla \zeta_j(x,\xi_i)^Ts(x), ~\text{for every } j\in \Lambda,~i\in \bar\Lambda.
	\end{equation}
	For every $j\in \Lambda$ and $i\in \bar\Lambda$, approximation of true  Hessian of $\zeta_j(x,\xi_i)$  is positive definite for every $x\in
	\mathbb{R}^n$. Also,    $W$ is a compact subset of $\mathbb{R}^n$ then, eigenvalues of approximation of true Hessian  of $\zeta_j(x,\xi_i)$ are  bounded on $C.$ So, there exists $\nu>0$ and $\bar{\nu},$ such that
	\begin{equation}\label{t^2}
		\bar{\nu}=\max_{x\in C,~ j\in \Lambda}\max_{i\in\bar{\Lambda}}\|\zeta_j(x,\xi_i)\|
	\end{equation}
	and \begin{equation}\label{t^3}
		\nu\|w\|^2\leq w^TH_j(x,\xi_i)w,~ \text{for every} ~j\in\Lambda,~i\in \bar \Lambda ~\text{and}~\forall x\in C,w\in\mathbb{R}^n.
	\end{equation}
	Now from $(\ref{t^1})$, $(\ref{t^2})$, $(\ref{t^3})$ and Cauchy-Schwartz inequality we have
	\begin{align*}
		\nu\|s(x)\|^2\leq\|\nabla\zeta_j(x,\xi_i)\|\|s(x)\|\leq\bar\nu \|s(x)\| ~\forall y\in C,~j\in\Lambda ~\text{and}~ i\in\bar{\Lambda}.
	\end{align*}
	Which implies that $$\|s(x)\|\leq \frac{\bar\nu}{\nu},~~ \forall x\in C$$
	i.e., $s(x)$, the Newton's directions are bounded on    $W$.\\
	Now for $x\in C$, $j\in \Lambda$ and $i\in\bar{\Lambda}$ we define $$E_{x,j,i}:C\to \mathbb{R}$$
	$$s.t.~~~~~~~~z\to\zeta_j(z,\xi_i)+ \nabla \zeta_j(z,\xi_{i})s(x)+\frac{1}{2}s(x)^T H_{j}(z,\xi_{i})s(x)-F_j(z).$$ Thus the family $\{E_{x,j,i}\}_{x\in C,~j\in\Lambda,i\in\bar{\Lambda}}$ is equicontinuous. Therefore, $\{\Delta_x=\displaystyle\max_{j\in\Lambda}\max_{i\in\bar{\Lambda}}E_{x,j,i}\}_{x\in C}$ is also equicontinuous. \\
	Now if we take $\epsilon>0$ there exists $\delta>0$ such that $\forall u,w\in C$ $$\|u-w\|<\delta\implies |\Delta_x(u)-\Delta_x(w)|<\epsilon, \forall x\in C.$$
	Hence, for $\|u-w\|<\delta,$
	\begin{align*}
		\vartheta(w)&\leq\zeta_j(w,\xi_i)+ \nabla\zeta_j(w,\xi_i)^Ts(u)+s(u)^T\nabla^2\zeta_j(w,\xi_i)^Ts(u)-F_{j}(w)\\
		&=\Delta_u(w)\\
		&\leq \Delta_u(u)+|\Delta_u(w)-\Delta_u(u)|\\
		&< \vartheta(u)+\epsilon,
	\end{align*}
	thus, we get $\vartheta(w)-\vartheta(u)<\epsilon$, also if we interchange $u$ and $w$ then we get $\vartheta(w)-\vartheta(u)>-\epsilon$. Therefore,
	$$|\vartheta(u)-\vartheta(w)|<\epsilon, ~\text{whenever}~ \|u-w\|<\delta.$$ Thus, $\vartheta(x)$ is continuous in    $W$. Since,    $W$ is any arbitrary compact subset of $\mathbb{R}^n$ and hence, $\vartheta(x)$ is continuous.
\end{proof}
To find the step size in the search direction, we develop an Armijo type inexact line search technique in the following subsection.
\subsection{Armijo type inexact line search for quasi Newton method }\label{ss3.2}
\begin{thm}\rm
	Assume that $x$ is not critical point of $F$. Then for any $\beta>0$  and $\epsilon\in(0,1]$ there exists an $\alpha\in[0,\epsilon]$ such that
	\begin{equation*}\label{18''}
		F_j(x+\alpha s)\leq  F_j(x)+  \alpha \beta \vartheta(x),
		\end{equation*}
\end{thm}
where $s(x)$ and $\vartheta(x)$ are given in the equation $(\ref{dd})$ and $(\ref{ov})$ respectively.
\begin{proof}
	By $(\ref{Thm 3.2})$, we have $\vartheta(x)\leq 0,$ and
		$$	\vartheta(x)=\max_{j\in\Lambda}\max_{i\in\bar{\Lambda}}  \{\zeta(x,\xi_i)+\nabla \zeta_j(x,\xi_i)^Ts(x) +\frac{1}{2}s(x)^TH_j(x,\xi_i)s(x)-F_j(x)\}.$$
	%According to our problem
	%$ F'_j(x,v)=\displaystyle \max_{i\in\bar\Lambda}\zeta(x,\xi_i)$ is %the directional derivative of $ F_j(x)$ at $x$ in direction $v$.  %$F''_j(x,v) = \displaystyle\max_{i\in Q_2(x)} v^T \nabla^2  %\zeta(x,\xi_i)^Tv$ is the second directional derivative of $F_j$ at %$x$ in direction $v$ which is not continuous. Hence to find the the step %size rule we require a continuous approximation of $F''_j(x,v)$. %Since $\zeta(x,\xi)$ is convex and twice continuously differentiable %and hence it is an upper uniformly  twice differentiable . From %(\citet{16}) Continuous approximation of $F''_j(x,v)$ is given by
	Now we consider an auxiliary function $F''^*_j(x,\alpha v)$ to find the step length size which is given by\begin{equation}\label{8}
		F''^*_j(x,\alpha s)= \max_{i\in\bar{\Lambda}}\{ \zeta_j(x,\xi_i) + \alpha\nabla \zeta_j(x,\xi_i)^Ts+\frac{1}{2}\alpha^2s^T H_j(x,\xi_i)^Ts \}-F_j(x)-F'_j(x,\alpha s),
	\end{equation}
	where $\alpha \in [0,\epsilon]$ and $\epsilon<1$. For every $j\in \Lambda$, (\ref{8}) can be written as
	$$F''^*_j(x,\alpha s) +F'_j(x,\alpha s)\leq \displaystyle\max_{i\in \bar\Lambda} \zeta_j(x,\xi_i) + \displaystyle\max_{i\in \bar\Lambda}\{\alpha \nabla \zeta_j(x,\xi_i)^Ts +\frac{1}{2}\alpha^2s^T H_j(x,\xi_i)^Ts\}-F_j(x),$$
	which implies
	\begin{equation}\label{9}
		F''^*_j(x,\alpha s)+F'_j(x,\alpha s) \leq\displaystyle\max_{i\in \bar\Lambda}\{\alpha \nabla \zeta_j(x,\xi_i)^Ts +\frac{1}{2}\alpha^2s^T H_j(x,\xi_i)^Ts\}.
	\end{equation}
	Since $\zeta_j(x,\xi_i)$, is twice continuously differentiable and convex. Then $\zeta_j(x,\xi_i)$ will be upper uniformly twice continuously differentiable (see \citet{bazaraa1982algorithm}, for the definition of uniformly differentiable function) and convex. Hence there exists $k^j_i>0$ such that
	\begin{center}
		$	\zeta_j(x+s,\xi_i) \leq \zeta_j(x,\xi_i)+\nabla \zeta_j(x,\xi_i)^Ts + \frac{1}{2} s^T H_j(x,\xi_i)^Ts +\frac{1}{3!}k^j_i\|s\|^3 $
	\end{center}
	since above inequality holds for every $s\in \mathbb{R}^n$ then it will be hold for $\alpha s$. Therefore,
	\begin{align*}
		\zeta_j(x+\alpha s,\xi_i) &\leq \zeta_j(x,\xi_i)+ \alpha\nabla \zeta_j(x,\xi_i)^Ts + \frac{1}{2}\alpha^2 s^T  H_j(x,\xi_i)^Ts +\frac{1}{3!}k^j_i\|\alpha s\|^3\\
		&\leq\max_{i\in \bar\Lambda}\{ \zeta_j(x,\xi_i)+ \alpha\nabla \zeta_j(x,\xi_i)^Ts + \frac{1}{2}\alpha^2 s^T H_j(x,\xi_i)^Ts +\frac{1}{3!}k^j_i\|\alpha s\|^3\}\\
		& \leq \max_{i\in \bar\Lambda}\{\zeta_j(x,\xi_i)+\alpha\nabla \zeta_j(x,\xi_i)^Ts + \frac{1}{2}\alpha^2 s^T H_j(x,\xi_i)^Ts\} +\frac{1}{3!}\alpha^3K^j\|s\|^3,
	\end{align*}
	where $K^j= \displaystyle \max_{i\in\bar{\Lambda}}{k^j_i}$. Now from equation (\ref{8}), we have
	\begin{equation*}
		\zeta_j(x+\alpha s,\xi_i) \leq 	F'_j(x,\alpha s) + F_j(x)+F''^*_j(x,\alpha s)+\frac{1}{3!} \alpha^3K^j \|s\|^3.
	\end{equation*}
Since above inequality hold for each $i\in \bar\Lambda$, therefore,
	\begin{equation}\label{z5}
		\max_{i\in\bar{\Lambda}}\zeta_j(x+\alpha s,\xi_i) \leq 	F'_j(x,\alpha s) + F_j(x)+F''^*_j(x, \alpha s)	+\frac{1}{3!} \alpha^3K^j \|s\|^3.
	\end{equation}
	Now, by (\ref{9})
	\begin{align}\label{3.15}
		F''^*_j(x,\alpha s)+F'_j(x,s) &\leq\displaystyle\max_{i\in \bar\Lambda}\{\alpha \nabla \zeta_j(x,\xi_i)^Ts +\frac{1}{2}\alpha^2s^TH_j(x,\xi_i)s\}\nonumber\\
	&	=\max_{i\in \bar\Lambda}\{\alpha\zeta_j(x,\xi_i)+\alpha \nabla \zeta_j(x,\xi_i)^Ts +\frac{1}{2}\alpha^2s^T  H_j(x,\xi_i)s-\alpha F_j(x)\}\nonumber\\
		&\leq\displaystyle\max_{j\in \Lambda}\max_{i\in \bar\Lambda}\{\alpha\zeta_j(x,\xi_i)+\alpha \nabla \zeta_j(x,\xi_i)^Ts +\frac{1}{2}\alpha^2s^T  H_j(x,\xi_i)s-\alpha F_j(x)\}\nonumber\\
		&\leq\alpha\max_{j\in\Lambda}\max_{i\in\bar{\Lambda}}  \{\zeta_{j}(x)+\nabla \zeta_j(x,\xi_i)^Ts +\frac{1}{2}s^TH_j(x,\xi_i)s-F_{j}(x)\}\nonumber\\&+\frac{\alpha(\alpha-1)}{2}\max_{j\in\Lambda}\max_{i\in\bar{\Lambda}}s^T H_j(x,\xi_i)^Ts\nonumber\\
		&\leq\alpha \theta (x)	+ \frac{1}{2}\alpha(\alpha-1) \displaystyle\displaystyle \max_{j\in \Lambda} \max_{i\in \bar{\Lambda}}  \{s^T H_j(x,\xi_i)^Ts\}.
	\end{align}
	%	\begin{equation}\label{10}
	%	F_j(x+\alpha v)\leq F'_j(x,\alpha v) + F_j(x)+F''^*_j(x,\alpha v)	+\frac{1}{3!} \alpha^3K^j \| v\|^3.
	%	\end{equation}
	Since $x$ is not a critical point then by Theorem \ref{Thm 3.2},
	%$(\ref{theta<0})$,
	 $(\ref{9})$, $(\ref{z5})$ and $(\ref{3.15}),$  we get
	\begin{equation}\label{11}
		F_j(x+ \alpha v)\leq F_j(x)+ \alpha \vartheta (x)	+ \frac{1}{2}\alpha(\alpha-1) \displaystyle\displaystyle \max_{j\in \Lambda} \max_{i\in \bar{\Lambda}}  \{s^T H_j(x,\xi_i)s\}+\frac{1}{3!}\alpha^3 K^j \|s\|^3.
	\end{equation}
	%The following  two results related to continuous approximation are true (\citet{16}):
	%\begin{enumerate}[(i)]
	%	\item 	$F''^*_{j}(x,\lambda v)\leq \lambdaF''^*_{j}(x,v),~\forall \lambda \in [0,1]$.
	%	\item $F_j(x+\lambda v)\leq  F_j(x) +\lambdaF'_j(x,v)+\frac{1}{2}\lambda^2F''^*_J(x,v)+ \lambda^3\frac{1}{3!}K \|v\|^2$
	%\end{enumerate}
	%with the help of (i) and (ii) and (8) we can write
	Since for each $i\in\bar{\Lambda}$ and $j\in\Lambda,$ $H_j(x,\xi_i)$ is positive definite  $\forall x\in\mathbb{R}^n.$ Then $s^T H_j(x,\xi_i)^Ts>0$, $0\not=s\in\mathbb{R}^n$.
	As $\alpha<1$, then the third term of R.H.S of $(\ref{11}),$ will be negative. Then $(\ref{11}),$ can be written as
	\begin{equation*}
		F_j(x+\alpha s)\leq  F_j(x)+ \alpha \vartheta (x)+ \alpha^3 \frac{1}{3!} K^j \|s\|^3.
	\end{equation*}
	If $\alpha>0$ is sufficiently small implies $\alpha^3\approxeq0$. Then, we have
	\begin{equation}\label{12*}
	F_j(x+\alpha s)\leq F_j(x)+ \alpha \vartheta (x)
	\end{equation}
	and also for some $\beta \in (0,1)$ we get
	\begin{equation}\label{12}
		\vartheta(x) < \beta \vartheta(x).
	\end{equation}
 Now, from (\ref{12*}) and (\ref{12}) we get
	\begin{equation}\label{18''}
		F_j(x+\alpha s)\leq  F_j(x)+  \alpha \beta \vartheta(x) .
	\end{equation}
	The inequality $(\ref{18''})$ represents the Armijo type inexact line search for quasi Newton's descent algorithm for the $RC^*_1$.
\end{proof}
\subsection{A modified BFGS update of quasi newton method for \texorpdfstring{$RC^*_1$}{Lg}}\label{ss3.3}
The BFGS update formula proposed by {\v{Z}}iga \citet{povalej2014quasi} is generally applicable for smooth multiobjective optimization problems. However, in the case of $RC^*_1$, which is a nonsmooth multiobjective optimization problem, there is no guarantee that the BFGS update (\ref{bfgsformop}) satisfies the curvature condition. Additionally, using the BFGS method to update $H_j(x^k, \xi_i)$ may fail to maintain the positive definiteness of $H_j(x^{k+1}, \xi_i)$.

Therefore, we need to modify the BFGS update (\ref{bfgsformop}) for $RC^*_1$. To do this, we introduce the concept of $H_j(x^{k+1}, \xi_i)$ for the problem $RC^*_1$. We assume that $H_j(x^{k}, \xi_i)$ is a symmetric positive definite matrix for each $\xi_i \in U$ and $j \in \Lambda$. Instead of using $p^k$ as in the traditional BFGS update, we consider a vector of the form:
 $$\eta^k=\sigma p^k+(1-\sigma)H(x^k,\xi_i)u^k,$$ $0 \leq\sigma \leq1$ i.e., closest to $p_k$ subject to the condition \begin{equation}\label{curvature}
	{u^k}^T\eta^k\geq0.2{u^k}^TH(x^k,\xi_i)u^k,
\end{equation}
where $u^k= x^{k+1}-x^k=\alpha_kv^k,$  $p^k=\nabla \zeta_j(x^{k+1},\xi_i)-\nabla \zeta_j(x^{k},\xi_i),$ $\alpha_{k}$ satisfy $(\ref{18''})$ at $x^k$ and the factor $0.2$ chosen based on empirical evidence. Thus the value of $\sigma$ is given by
\begin{equation}
	\sigma =
	\left\{
	\begin{array}{lr}
		1, & \text{if } {u^k}^Tp^k \geq 0.2{u^k}^TH(x^k,\xi_i)u^k,\\
		\frac{0.8{u^k}^TH(x^k,\xi_i)u^k}{{u^k}^TH(x^k,\xi_i)u^k-{u^k}^Tp^k}, & \text{if } {u^k}^Tp^k < 0.2{u^k}^TH(x^k,\xi_i)u^k
	\end{array}
	\right\}.
\end{equation}
Thus the new modified BFGS update formula for $RC^*_1$ is given by
\begin{equation}\label{*1}
	H_j(x^{k+1},\xi_i)=H_j(x^k,\xi_i)-\frac{H_j(x^k,\xi_i)u^k(u^k)^TH_j(x^k,\xi_i)}{(u^k)^TH_j(x^k,\xi_i)u^k}+\frac{\eta^k(\eta^k)^T}{(u^k)^T \eta^k}.
\end{equation}
After multiplying (\ref{*1}) by $u^k$ we get
$$H_j(x^{k+1},\xi_i)u^k=H_j(x^k,\xi_i)u^k-\frac{H_j(x^k,\xi_i)u^k(u^k)^TH_j(x^k,\xi_i)u^k}{(u^k)^TH_j(x^k,\xi_i)u^k}+\frac{\eta^k{u^k}^T\eta^k}{(u^k)^T \eta^k},$$
which implies,
 \begin{equation}\label{**1}
	H_j(x^{k+1},\xi_i)u^k=\eta^k.
	\end{equation}
Equation (\ref{**1}), represents the quasi Newton condition or secant equation. After multiply  (\ref{**1}) by $u^k$ and with the help of (\ref{curvature}) we get
\begin{equation}\label{***1}
	{u^k}^TH_j(x^{k+1},\xi_i)u^k={u^k}^T\eta^k>0.
	\end{equation}
Equation (\ref{***1}) represents the curvature condition which ensures that  $H_j(x^{k+1},\xi_i)$ is positive definite for each $\xi_i$ and $j$. Thus, we have new modified BFGS update formula (\ref{*1}) for Hessian update which satisfies quasi Newton condition and curvature condition.
\par With the derived quasi Newton descent direction, step length size, and the modified BFGS update formula for updating the Hessian, we are now equipped to present the quasi Newton algorithm for solving $RC^*_1$. Using these components, we can outline the steps of the algorithm as follows:
%	\begin{alg1}\label{algo1}\rm(\textbf{Quasi Newton algorithm for $RC^*_1$})
%	\begin{enumerate}[{Step} 1]
%		
%			\item\label{th0} (\textbf{Initialization})
%			\begin{enumerate}
%				\item Choose $\epsilon>0$, $\beta \in (0,1) $   and $x^0\in \mathbb{R}^n$ and  $H_j(x^0,\xi_i)$ is symmetric and positive definite matrix for each $\xi_i$. Set $k :=0.$
%				\item\label{1b} Solve $Q_N(x)$ at $x^0$ to obtain $s^0=s(x^0)$ and $\vartheta(x^0)$. If $|\vartheta(x^0)|<\epsilon$ or $\|s(x^0)\|<\epsilon$ then stop. Otherwise, choose $\alpha^0$ such that, $\alpha^0$ satisfies $(\ref{18''})$ and calculate $x^1=x^0+\alpha_0s^0,$ set $k=1$ and go to Step \ref{step2}.
%			\end{enumerate}
%		\item\label{step2} (\textbf{Main loop})
%		\begin{enumerate}
%		\item\label{2a} Solve $Q_N(x)$ at $x^k$ to obtain $s^k=s(x^k).$
%		\item \label{2b} Choose $\alpha_k$ as the largest $\alpha \in \{ \frac{1}{2^r} : r=1,2,3,\ldots\} $ satisfying (\ref{18''}).
%		\item\label{2c} Define $x^{k+1}:= x^k + \alpha_{k} v^k$. Update $H_j(x^{k},\xi_i)$ with the help of $(\ref{*1})$. Set $k:=k+1$, go to Step $\ref{2a}$ and repeat this process until we do not get approximate critical point.
%	\end{enumerate}
%\end{enumerate}
%\end{alg1}
\begin{alg1}\label{algo1}\rm(\textbf{Quasi Newton algorithm for $RC^*_1$})
	%\begin{enumerate}[{Step} 1]
		
		%\item\label{th0} (\textbf{Initialization})
		\begin{enumerate}[{Step} 1]
			\item Choose $\epsilon>0$, $\beta \in (0,1) $   and $x^0\in \mathbb{R}^n$ and  $H_j(x^0,\xi_i)$ is symmetric and positive definite matrix for each $\xi_i$ and $j$. Set $k :=0.$
			\item\label{1b} Solve $Q_N(x^k)$ at $x^k$ to obtain $s^k=s(x^k)$ and $\vartheta(x^k)$. If $|\vartheta(x^k)|<\epsilon$ or $\|s(x^k)\|<\epsilon$ then stop. Otherwise, proceed to Step \ref{step2}.
			\item\label{step2} Choose $\alpha_k$ as the largest $\alpha \in \{ \frac{1}{2^r} : r=1,2,3,\ldots\},$ which satisfies (\ref{18''}).
			\item\label{2c} Set $x^{k+1}:= x^k + \alpha_{k} s^k$. Update: Generate a positive definite matrix $H_j(x^{k+1},\xi_i)$ with the help of $(\ref{*1})$. Set $k:=k+1$, go to Step $\ref{1b}$ and repeat this process until we do not reach the stopping criteria.
		\end{enumerate}
%	\end{enumerate}
\end{alg1}
Upon examining Algorithm \ref{algo1}, we can observe that at any given point $x^k \in \mathbb{R}^n$, $Q_N(x^k)$ is a convex programming quadratic subproblem due to the positive definiteness of $H_j(x^k, \xi_i)$ for each $j$ and $i$. As a result, the unique solution of $Q_N(x^k)$, denoted as $s(x^k)$, exists, ensuring that Step \ref{1b} is well-defined. According to Theorem \ref{Thm 3.2}, the stopping criterion for Algorithm \ref{algo1} can be set as $\|s(x^k)\| < \epsilon$ or $|\vartheta(x^k)| < \epsilon$. If, at iteration $k$, the stopping criterion is not met (i.e., $x^k$ is not an approximate critical point), we proceed to Step \ref{step2}. In this step, we find $\alpha_k$ such that $\alpha_k$ is the largest $\alpha \in \{\frac{1}{2^r} : r = 1, 2, 3, \ldots\}$ satisfying equation (\ref{18''}). Next, we calculate $x^{k+1} = x^k + \alpha_k s^k$ in Step \ref{2c}, and update $H_j(x^k, \xi_i)$ to obtain $H_j(x^{k+1}, \xi_i)$ using equation (\ref{*1}). The updated $H_j(x^{k+1}, \xi_i)$ is a symmetric and positive definite matrix for each $j$ and $\xi_i$. Finally, we set $k := k + 1$ and return to Step \ref{1b}, repeating this process until an approximate critical point is reached. There is one more thing to know about Algorithm \ref{algo1}. When we reach an approximative critical point after reaching the stopping criteria, that critical point will be the approximative weak Pareto optimal solution of $RC_1^*.$
\subsection{Convergence analysis of algorithm \ref{algo1}}\label{ss3.4}
Initially, we prove Lemma \ref{lm2}, which helps to prove the convergence of the sequence generated by Algorithm \ref{lm2}.
\begin{lem}\label{lm2} If $\vartheta(x)$ is given by (\ref{dd}), then
	for all $k=0,1,2,\ldots$ and $j\in\Lambda$,
	$$	\sum\limits^k_{r=0}\alpha_{r}\vert\vartheta(x^r)\vert\leq \beta^{-1}(F_{j}(x^0)-F_{j}(x^{k+1})),$$
	where  $F_{j}(x^k)=\displaystyle \max_{i \in \bar \Lambda} \zeta_j(x^k,\xi_i).$
\end{lem}
\begin{proof}
Assume that $\{x^k\}$ is a sequence generated by Algorithm \ref{algo1}. Then, by Step~\ref{step2} of Algorithm \ref{algo1}
\begin{equation*}
	-\alpha_k  \vartheta(x^k)	\leq \beta^{-1}(F_j(x^{k}) -F_j(x^{k+1})).
\end{equation*}
Thus we have
\begin{equation*}\label{19'}
	-\sum\limits^k_{r=0}\alpha_r \vartheta(x^r)\leq \beta^{-1}(F_j(x^{0}) -F_j(x^{k+1})).
\end{equation*}
Also
\begin{equation}\label{19'}
	\sum\limits^k_{r=0}\alpha_r \vartheta(x^r)\geq -\beta^{-1}(F_j(x^{0}) -F_j(x^{k+1})).
\end{equation}
Since, we know $\vartheta(x^r)<0,$ then equation $(\ref{19'})$, can be written as
\begin{equation}\label{20'}
	\beta^{-1}(F_j(x^{0}) -F_j(x^{k+1}))	\geq \sum\limits^k_{r=0}\alpha_r \vartheta(x^r).
\end{equation}
Now, from $(\ref{19'})$ and $(\ref{20'})$
\begin{equation}\label{lll}
	\sum\limits^k_{r=0}\alpha_r \vert\vartheta(x^r)\vert\leq  \beta^{-1}(F_j(x^{0}) -F_j(x^{k+1}))
\end{equation}
hence the lemma.
\end{proof}
\begin{thm}\label{3.5}\rm
Let $s(x)$ and $\vartheta(x)$ be given by (\ref{dd}) and (\ref{ov}), respectively, which represent the optimal value and optimal solution of the subproblem $Q_N(x)$. At any given point $x^k\in \mathbb{R}^n$, in the direction $s(x^k)$ the Armijo type inexact line search is given by
\begin{equation}\label{18'''}
	F_j(x^k+\alpha s^k)\leq  F_j(x^k)+  \alpha_k \beta \vartheta(x^k).
\end{equation}
Then
\begin{enumerate}
	\item\label{c2} if $\{x^k\}$ is a sequence generated by Algorithm \ref{algo1}. Then accumulation point of $\{x^k\}$ is a Pareto point for $F$;
\item\label{c3} moreover, if $x^0$ lies in a compact level-set of $F$, and $\{x^k\}$ be a sequence generated by  Algorithm $\ref{algo1}$. Then $\{x^k\}$ converges to a Pareto optimum $x^*\in \mathbb{R}^n$ for $F$ for the problem $RC^*_1$.
\end{enumerate}
\end{thm}
\begin{proof}
	To prove \ref{c2}, we assume  $x^*$ is an accumulation point  of the sequence $\{x^k\}$ generated by Algorithm $\ref{algo1}$. Then there exists a subsequence $\{x^{k_l \}}$ such that
\begin{equation*}
	\lim_{l\to\infty}x^{k_l}  = x^*.
\end{equation*}
Using the lemma \ref{lm2}, with $l = k_l$ and taking the limit $l\to\infty$ from the continuity of $F_{j}$\\
\begin{equation*}
	\sum\limits^\infty_{l=0} \alpha_l |\vartheta(x^l)| < \infty.
\end{equation*}
Therefore
\begin{equation*}
	\lim_{k\to\infty}\alpha_{k} \vartheta(x^k)  = 0.
\end{equation*}
And in particular, we have
\begin{equation}\label{18****}
	\lim_{l\to\infty}\alpha_{k_l} \vartheta(x^{k_l})  = 0.
\end{equation}
Suppose that $x^*$ is non critical point, then
\begin{equation}\label{19****}
	\vartheta(x^*) < 0 ~~and ~~s^*\not=0.
\end{equation}
Now define
\begin{equation*}
	G:\mathbb{R}^m \to \mathbb{R},  ~G(u)=\displaystyle \max_{j\in \Lambda}u_j.
\end{equation*}
Using the definition of $\vartheta(.)$ and descent direction $s^*$, we can conclude that there exists a non negative integer $q$ such that
\begin{equation*}
	G(F(x^*+2^{-q}s^*)-F(x^*))<\beta 2^{-q}\vartheta(x^*).
\end{equation*}
Since $G(.)$ and $\vartheta(x)$ are continuous, then
\begin{equation*}
	\lim_{l\to\infty}s^{k_l}=s^*, ~~~\lim_{l\to\infty} \vartheta(x^{k_l})=\vartheta(x^*).
\end{equation*}
And for $l$ large enough
\begin{equation}
	G(F(x^{k_l}+2^{-q}s^{k_l})-F(x^{k_l}))<\beta 2^{-q}\vartheta(x^{k_l}),
\end{equation}
which implies
\begin{equation}\label{25'}
	F_j(x^{k_l}+2^{-q}s^{k_l})-F_j(x^{k_l})<\beta 2^{-q}\vartheta(x^{k_l}).
\end{equation}
By the Algorithm $\ref{algo1}$,  $Step ~\ref{step2}$ we have
\begin{equation}\label{26'}
	F_j(x^{k_l}+\alpha_{k_l}s^{k_l})\leq F_j(x^{k_l})+\beta \alpha_{k_l}\vartheta(x^{k_l}).
\end{equation}
Since $\alpha_{k_l}$ is the largest of $\{\frac{1}{2^n}:n=1,2,3,....\}$, and $\sigma \in (0,1)$. Then by equation $(\ref{25'})$ and $(\ref{26'})$ we have
\begin{equation}
	\beta \alpha_{k_l}\vartheta(x^{k_l}) \leq \beta 2^{-q}\vartheta(x^{k_l}).
\end{equation}
Which implies
\begin{equation}
	\alpha_{k_l} \geq 2^{-q}>0, ~~\text{for}~l~\text{large enough.}
\end{equation}
Hence
\begin{equation}
	\lim_{l\to\infty}\alpha_{k_l} \vartheta(x^{k_l}) > 0.
\end{equation}
Which contradicts the equation $(\ref{18****})$. Therefore, the condition in equation $(\ref{19****})$ is not true. Thus $x^*$ is a critical point. And  $F$ is strictly convex, so $x^*$ is a Pareto optimum, hence the item \ref{c2}. \\
Now to prove the item \ref{c3}, let $L_0$ be the $F(x^0)$-level set of  $F$, i.e.,
\begin{equation*}
	L_0=\{x\in \mathbb{R}^n:F(x)\leq F(x^0)\}.
\end{equation*}
Since the sequence ${F(x^k)}$ is $\mathbb{R}^m$-decreasing, we have $x^k \in L_0$ for all $k$. Therefore, $\{x^k\}$ is bounded, and all its accumulation points lie in $L_0$. By applying the above theorem, we can conclude that these accumulation points are all Pareto optima for $F$.
Furthermore, since for every $x^k \in L_0$ we have $F(x^{k+1}) \leq F(x^k),$ we can use Lemma 3.8 in \citet{fukuda2014survey} to show that for any accumulation point $x^*$ of $\{x^k\}$, we have $F(x^*) \leq F(x^k)$ and $\lim_{k \to \infty} F(x^k) = F(x^*)$. Additionally, $F(x)$ is constant among the accumulation points of $\{x^k\}$.
Since $F$ is strongly convex, meaning it is strictly $\mathbb{R}^m$-convex, there exists at least one accumulation point of $\{x^k\}$, denoted as $x^*$. Thus, the proof is complete.

\end{proof}

\subsection{Convergence rate of Algorithm \ref{algo1}}\label{ss3.5}

Under the usual assumptions, we will prove the convergence of the infinite sequence generated by Algorithm \ref{algo1} to the Pareto optimum with a superlinear rate when full quasi Newton steps are performed, i.e., $\alpha_k = 1$ for sufficiently large $k$. First, we assume a bound for the descent direction $s(x)$. Next, we prove a bound for $\vartheta(x^{k+1})$ using information from the previous iteration point $x^k$. Finally, we show that full quasi Newton steps are performed when $k$ is large enough, i.e., to say $\alpha_{k}=1$.
%\textbf{Assumption 1}
%For any $x\in\mathbb{R}^n$ we have
%\begin{equation}\label{30}
%	\|v(x)\|^2\leq \frac{2}{\omega}\mid\theta(x)\mid,
%\end{equation}

%where $\omega>0$ and defined in $\ref{1**}$.\\
\begin{thm}\rm Let $x^k$ be  a sequence generated by Algorithm $\ref{algo1}$.
	\begin{enumerate}
		\item\label{cr1} If for any $k$, $\exists ~\lambda_{ij}$ such that 	$$\sum\limits_{j\in \Lambda}  \sum\limits_{i\in\bar\Lambda} \lambda_{ij}=1,$$~~~  then $$\vartheta(x^{k+1})\geq -\frac{1}{2\omega}\|\sum\limits_{j\in\Lambda}\sum\limits_{i\in \bar\Lambda}\lambda^k_{ij}\nabla \zeta_j(x^{k+1},\xi_i)\|^2$$ where $\omega>0$ defined in Definition $\ref{1**}$.
		\item\label{cr2} If we assume for any $x\in\mathbb{R}^n$,
		\begin{equation}\label{30}
			\|s(x)\|^2\leq \frac{2}{\omega}\mid\vartheta(x)\mid,
		\end{equation}
		
		where $\omega>0$ defined in Definition $\ref{1**},$
		and let $x^0\in\mathbb{R}^n$ be in compact level set of $F$. Then $\{x^k\}$ converges to a Pareto optimum point $x^*$ moreover $\alpha_k=1$ for $k$ large enough and rate of convergence of $\{x^k\}$ to  $x^*$ is superlinear.
	\end{enumerate}
\end{thm}
\begin{proof}
	Assume the sequence $\{x^k\}$ is generated by Algorithm $\ref{algo1}$. By $(\ref{ov})$ and $(\ref{2})$ we have
	
\begin{align*}
		\vartheta(x^{k+1})&=\sum\limits_{j\in \Lambda}  \sum\limits_{i\in\bar\Lambda} \lambda^k_{ij}\max_{j\in\Lambda}\max_{i\in \bar\Lambda}  \{\nabla \zeta_j(x^{k+1},\xi_i)^Ts(x^{k+1}) +\frac{1}{2}v(x^{k+1})^TH_j(x^{k+1},\xi_i)s(x^{k+1})\}\\
		&	\geq\sum\limits_{j\in \Lambda}  \sum\limits_{i\in\bar\Lambda} \lambda^k_{ij} \{\nabla \zeta_j(x^{k+1},\xi_i)^Ts(x^{k+1}) +\frac{1}{2}s(x^{k+1})^TH_j(x^{k+1},\xi_i)s(x^{k+1})\}\\
		&	\geq\displaystyle\min_{s}\sum\limits_{j\in \Lambda}  \sum\limits_{i\in\bar\Lambda} \lambda^k_{ij} \{\nabla \zeta_j(x^{k+1},\xi_i)^Ts +\frac{1}{2}s^TH_j(x^{k+1},\xi_i)s\}.
	\end{align*}
	Therefore,
	\begin{equation}\label{31}
		\vartheta(x^{k+1}) \geq  \displaystyle\min_{s}\sum\limits_{j\in \Lambda}  \sum\limits_{i\in\bar\Lambda} \lambda^k_{ij} \big( \nabla \zeta_j(x^{k+1},\xi_i)^Ts
		+\frac{\omega}{2}\|s\|^2 \big).
	\end{equation}
	Solving the R.H.S. of inequality $(\ref{31})$  we get
	\begin{equation}\label{32}
		\displaystyle\min_{s}\sum\limits_{j\in \Lambda}  \sum\limits_{i\in \bar\Lambda} \lambda^k_{ij} \big( \nabla \zeta_j(x^{k+1},\xi_i)^Ts
		+\frac{1}{2}\|s\|^2 \big)=-\frac{1}{2\omega}\|\sum\limits_{j\in\Lambda}\sum\limits_{i\in\bar\Lambda}\lambda^k_{ij}\nabla \zeta_j(x^{k+1},\xi_i)\|^2
	\end{equation}
	then by $(\ref{31})$ and $(\ref{32})$
	\begin{equation}\label{33}
		\vartheta(x^{k+1})\geq -\frac{1}{2\omega}\|\sum\limits_{j\in\Lambda}\sum\limits_{i\in\bar\Lambda}\lambda^k_{ij}\nabla \zeta_j(x^{k+1},\xi_i)\|^2.
	\end{equation}
	Hence the proof of item \ref{cr1}.\\
	Now we prove item $\ref{cr2}$. Since $\{x^k\}$ be a sequence  generated by Algorithm $\ref{algo1}$ whose initial point $x^0$ belongs to compact level set of $F.$ Therefore by Theorem \ref{3.5}, $\{x^k\}$ converges to $x^*.$ Here we prove $\{x^k\}$ converges to $x^*$ with superlinear rate. Since $x^*$ is critical point then by the definition there is no descent direction at $x^*$ i.e., $s(x^*)=0.$
	Since  $\zeta_{j}(x,\xi_i)$ is twice continuously differantiable for each $x$ and $\xi_{i},$ we get\\
	\begin{align*}
		\zeta_j(x^k+v^k,\xi_i)&\leq \zeta_j(x^k,\xi_i)+\nabla\zeta_j(x^k,\xi_i)^Ts^k+\frac{1}{2}(s^k)^TH_j(x^k,\xi_i)(s^k)+\frac{\epsilon}{2}\|s^k\|^2\\
		&\leq \displaystyle \max_{i\in \bar\Lambda}\{\zeta_j(x^k,\xi_i)+\nabla\zeta_j(x^k,\xi_i)^Ts^k+\frac{1}{2}(s^k)^TH_j(x^k,\xi_i)(s^k)+\frac{\epsilon}{2}\|s^k\|^2\}\\
		&\leq \displaystyle \max_{i\in \bar\Lambda}\zeta_j(x^k,\xi_i)+\displaystyle \max_{j\in \Lambda}\max_{i\in \bar\Lambda}\{\nabla\zeta_j(x^k,\xi_i)^Ts^k+\frac{1}{2}(s^k)^TH_j(x^k,\xi_i)(s^k)\}+\frac{\epsilon}{2}\|s^k\|^2.
	\end{align*}
	Then
	\begin{align*}
		&\zeta_j(x^k+s^k,\xi_i)\leq F_j(x^k)+\vartheta(x^k)+\frac{\epsilon}{2}\|s^k\|^2~ \text{holds for each}~ \xi_i.
	\end{align*}
 Then
	\begin{align*}
		&\displaystyle\max_{i\in \bar\Lambda} \zeta_j(x^k+s^k,\xi_i)\leq F_j(x^k)+\vartheta(x^k)+\frac{\epsilon}{2}\|s^k\|^2
	\end{align*}
	\begin{align*}
	\text{and}~~~~	&F_j(x^k+s^k) -F_j(x^k)\leq\vartheta(x^k)+\frac{\epsilon}{2}\|s^k\|^2.
	\end{align*}
	Then by $(\ref{30})$,
	\begin{align*}
		F_j(x^k+s^k) -F_j(x^k)&\leq\sigma\theta(x^k)+(1-\sigma)\vartheta(x^k)+\frac{\epsilon}{2}\|s^k\|^2\\
		&\leq\sigma\vartheta(x^k)+(\epsilon-\omega(1-\sigma))\frac{\|s^k\|^2}{2},~ \text{for all} ~k\geq k_\epsilon.
	\end{align*}
	This inequality shows that if $\epsilon<\omega(1-\sigma)$, then by Algorithm $\ref{algo1}$, in $Step ~\ref{step2}$, ~$\alpha_{k}=1$ is accepted for $k\geq k_\epsilon.$\\
	For superlinear convergence suppose that $ \epsilon<\omega(1-\sigma).$ Use the Taylor's first order expansion of $\sum\limits_{j\in\Lambda}\sum\limits_{i\in \bar{\Lambda}}\lambda^k_{ij}\nabla \zeta_j(x^{k+1},\xi_i)$ at $x^k$, in $s^k$, we conclude that for any $k\geq k_\epsilon$, it holds
	
	$$ \| \sum\limits_{j\in\Lambda}\sum\limits_{i\in \bar{\Lambda}}\lambda^k_{ij}\nabla \zeta_j(x^{k+1},\xi_i)\|=\|\sum\limits_{j\in\Lambda}\sum\limits_{i\in \bar{\Lambda}}\lambda^k_{ij}\nabla \zeta_j(x^{k+1},\xi_i)-0\|.$$
	By  $(\ref{3}),$
	\begin{align*}
		\| \sum\limits_{j\in\Lambda}\sum\limits_{i\in \bar{\Lambda}}\lambda^k_{ij}\nabla \zeta_j(x^{k+1},\xi_i)\|&=\|\sum\limits_{j\in\Lambda}\sum\limits_{i\in \bar{\Lambda}}\lambda^k_{ij}\nabla \zeta_j(x^{k+1},\xi_i)-\big(\sum\limits_{j\in \Lambda}  \sum\limits_{i\in \bar{\Lambda}} \lambda^k_{ij} \nabla \zeta_j(x^k,\xi_i)+ \sum\limits_{j\in \Lambda}  \sum\limits_{i\in \bar{\Lambda}}\lambda^k_{ij} H_j(x^k,\xi_l)s^k\big)\|\\
		&\leq\|\sum\limits_{j\in\Lambda}\sum\limits_{i\in \bar{\Lambda}}\lambda^k_{ij}\nabla \zeta_j(x^k+s^k,\xi_l)
		-(\sum\limits_{j\in \Lambda}  \sum\limits_{i\in \bar{\Lambda}} \lambda^k_{ij} \nabla \zeta_j(x^k,\xi_l)\\&+ \sum\limits_{j\in \Lambda}  \sum\limits_{i\in \bar{\Lambda}}\lambda^k_{ij} H_j(x^k,\xi_i)^Ts^k)\|\\
		&\leq \sum\limits_{j\in \Lambda}  \sum\limits_{i\in \bar{\Lambda}} \lambda^k_{ij}\|\epsilon s^k\|.
	\end{align*}
	Thus we get
	\begin{equation}\label{34}
		\|\sum\limits_{j\in\Lambda}\sum\limits_{i\in i\in \bar{\Lambda}}\lambda^k_{ij}\nabla \zeta_j(x^{k+1},\xi_i)\|\leq \epsilon \| s^k\|
	\end{equation}
	By $(\ref{33})$ and $(\ref{34})$ we get
	 \begin{equation}\label{35}
	\vartheta(x^{k+1})\geq -\frac{1}{2\omega}(\epsilon \|s^k\|)^2~~\text{and}~~	\mid\vartheta(x^{k+1})\mid\leq\frac{1}{2\omega}(\epsilon \|s^k\|)^2.
	\end{equation}
	Since $\alpha_k=1$ and $x^{k+1}=x^k+s^k,$ from $(\ref{30})$ and $(\ref{35}$) we have $$\|x^{k+1}-x^{k+2}\|=\|s^{k+1}\|\leq \frac{\epsilon}{\omega}\|s^k\|.$$
	Thus if $k\geq k_\epsilon$ and $j\geq 1$ then
	\begin{equation}\label{36}
		\|x^{k+j}-x^{k+j+1}\|\leq \big(\frac{\epsilon}{\omega}\big)^j\|x^k-x^{k+1}\|.
	\end{equation}
	To prove the superlinear convergence rate take $0<\aleph<1$ and define $$\epsilon^*=\min\{1-\sigma,\frac{\aleph}{1+2\aleph}\}\omega.$$
	If $\epsilon<\epsilon^*$ and $k\geq k_\epsilon,$ then by $(\ref{36})$ and convergence of $\{x^k\}$ to $x^*$, we have
	\begin{align*}
		\|x^*-x^{k+1}\|&\leq\sum\limits^{\infty}_{j=1}\|x^{k+j}-x^{k+j+1}\|\\
		&\leq\sum\limits^{\infty}_{j=1}(\frac{\aleph}{1+2\aleph})^j\|x^k-x^{k+1}\|\\
		&=(\frac{ \aleph}{1+\aleph})\|x^k-x^{k+1}\|.
	\end{align*}\
	\begin{equation}\label{37}
		\|x^*-x^{k+1}\|\leq(\frac{ \aleph}{1+\aleph})\|x^k-x^{k+1}\|.
	\end{equation}
	Now with the help of triangle inequality and $(\ref{37})$ we get
	\begin{align*}
		\|x^*-x^{k}\|&\leq\|x^k-x^{k+1}\|-\|x^{k+1}-x^*\|\\
		&\leq\|x^k-x^{k+1}\|-(\frac{ \aleph}{1+\aleph})\|x^k-x^{k+1}\|
	\end{align*}
	therefore \begin{equation}\label{38}
		\|x^*-x^{k}\|\leq(\frac{1}{1+\aleph})\|x^k-x^{k+1}\|.
	\end{equation}
	Now by $(\ref{37})$ and $(\ref{38}),$ we conclude if $\epsilon<\epsilon^*$ and $k\geq k_\epsilon$, then
	
	\begin{equation*}
		\|x^*-x^{k+1}\|\leq\aleph\|x^*-x^k\|,
	\end{equation*}
 \begin{equation*}
		\text{i.e.,}~~~\frac{\|x^*-x^{k+1}\|}{\|x^*-x^k\|}\leq\aleph.
	\end{equation*}
	Since $\aleph$ was arbitrary in $(0,1)$, thus the above quotient tends to zero and hence $\{x^k\}$ converges to $x^*$ with superlinear rate.
\end{proof}
\subsection{Numerical results}\label{ss3.6}
In this subsection, we consider some numerical examples, and solve with the help of quasi Newton algorithm (Algorithm \ref{algo1}) and compare with weighted sum method. To solve numerical examples with the help of Algorithm \ref{algo1}, a $Python$ code is developed. In our computations, we consider $\epsilon=10^{-4}$ as tolerance or maximum $5000$ iterations is considered as stopping criteria. Thus, in our computations, we use $\|s^k\|<10^{-4}$ or $|\vartheta|<10^{-4}$ or maximum $5000$ number of iterations as a stopping criteria.
\par Since we know the solution of multiobjective optimization problem is not an isolated point but a set of Pareto optimal solutions. We used a multi-start technique to generate a well distributed approximation of Pareto front. Here 100 uniformly distributed random point is chosen between $lb$ and $ub$ (where $lb,ub\in\mathbb{R}^n$ and $lb<ub$) and Algorithm \ref{algo1} is executed separately. The nondominated set of collection of critical points is considered as an approximate Pareto front. We have compared the approximate Pareto fronts obtained by Algorithm \ref{algo1}  with the approximate Pareto fronts obtained by weighted sum method. In weighted sum method, we solve the following single objective optimization problem
$$\min_{x\in\mathbb{R}^n} \left( w_1\phi_1(x)+ w_2\phi_2(x)+\dots+ w_m\phi_m(x)\right),$$ where $( w_1, w_2,\dots, w_m)= w\geq 0$, $ w\neq0$ using the technique developed in \citet{bazaraa1982algorithm} with initial approximation $x^0=\frac{1}{2}(lb+ub)$. For bi-objective optimization problems, we have considered weights $(1,0)$, $(0,1)$, and 98 random weights uniformly distributed in the square area of~~ $[0,1]\times[0,1]$ (i.e., any 98 random weights uniformly distributed in this area). On the other hand, for three-objective optimization problems, we have considered four types of weights: $(1,0,0)$, $(0,1,0)$, $(0,0,1)$, and 97 random weights uniformly distributed in the cubic area of $[0,1]\times[0,1]\times[0,1]$ (i.e., any 97 random weights uniformly distributed in this area).
\par To find the descent direction at $k^{th}$ iteration we solve subproblem 	\begin{center}
	$Q_N(x^k):$ $\displaystyle \min_{(t,s)\in \mathbb{R}\times\mathbb{R}^n}t$
	
	$ s.t. ~~~~~~~ \zeta_j(x^k,\xi_i)+ \nabla \zeta_j(x^k,\xi_i)^Ts +\frac{1}{2}s^TH_j(x^k,\xi_i)s-F_j(x)\leq t, ~~~\forall i\in \bar{\Lambda} ~\text{and} ~j\in \Lambda,$\\
	$~~~~~~~~~~~~lb\leq x+s\leq ub.$\\
	 For step length size $\alpha_k,$ we choose $\alpha_k,$ as largest $\alpha \in \{ \frac{1}{2^r} : r=1,2,3,\ldots\},$ which satisfies (\ref{18''}) i.e., \begin{equation*}
		F_j(x^k+\alpha_k s)\leq  F_j(x^k)+  \alpha_k \beta \vartheta(x^k).
		\end{equation*}
\end{center}
For Positive definite Hessian update we use modified BFGS formula which is given in (\ref{*1}).
\begin{example}\label{example1}\rm\textbf{(Two dimensional bi-objective non-convex problem under uncertainty set of two elements)}\rm\\
	Consider the uncertain bi-objective optimization problem
	$$P(U)=\{P(\xi):\xi\in U\}$$ such that
	$$P(\xi):~~~~~~ \displaystyle\min_{x\in\mathbb{R}^2} \zeta (x,\xi),$$
	where $\zeta:\mathbb{R}^2 \times U\to \mathbb{R}^2$ such that 	$\zeta(x,\xi)=(\zeta_1(x,\xi),\zeta_2(x,\xi)).$\\
	Consider  $U=\{(1,2),~(2,3)\}\subset \mathbb{R}^2$ such that $\xi^1=(\xi^1_1,\xi^1_2)=(1,2)$ $\xi^2=(\xi^2_1,\xi^2_2)=(2,3),$ and $\bar{\Lambda}=\{i:\xi_i\in U\}=\{1,2\}.$
	$$\zeta(x,\xi^i)=\bigg(\frac{1}{4}\big((x_1-\xi^i_1)^4+2(x_2-\xi^i_2)^4\big),~ (\xi^i_1x_2-\xi^i_2x_1^2)^2+(1-\xi^i_1x_1)^2\bigg),$$
	\begin{equation*}
		\text{and}~~	\zeta_1(x,\xi^i)=\frac{1}{4}(x_1-\xi^i_1)^4+2(x_2-\xi^i_2)^4,~	\zeta_2(x,\xi^i)=(\xi^i_1x_2-\xi^i_2x_1^2)^2+(1-\xi^i_1x_1)^2,~i=1,~2.
	\end{equation*}
Objective wise worst case cost type robust counterpart to $P(U)$ is given by
	\begin{equation}
		RC^*_1:~~~~~	\displaystyle\min_{x\in \mathbb{R}^2}F(x),\end{equation}
	\begin{equation*}
		\text{where}~F(x)=(F_1(x), F_2(x))~\text{and}
	\end{equation*}
	\begin{equation*}
		F_1(x)=\displaystyle \max_{i\in\{1,2\}} \zeta_1(x,\xi^i)=\displaystyle \max\{\frac{1}{4}\big((x_1-1)^4+2(x_2-2)^4\big),~\frac{1}{4}\big((x_1-2)^4+2(x_2-3)^4\big)\},
	\end{equation*}
	\begin{equation*}F_2(x)=\displaystyle \max_{i\in\{1,2\}} \zeta_2(x,\xi^i)=\displaystyle \max\{(x_2-2x_1^2)^2+(1-x_1)^2,(2x_2-3x_1^2)^2+(1-2x_1)^2\}.
	\end{equation*}
\begin{figure}[!httb]
	\centering
	%	\begin{subfigure}[b]{0.45\textwidth}
	%	\centering
	%	\includegraphics[width=\textwidth]{w55}
	%	\caption{Approximate Pareto front generated by Newton's algorithm for $RC^*_1$}
	%	\label{fig:y equals x}
	%	\end{subfigure}
	\hfill
	\begin{subfigure}[b]{0.45\textwidth}
		\centering
		\includegraphics[width=\textwidth]{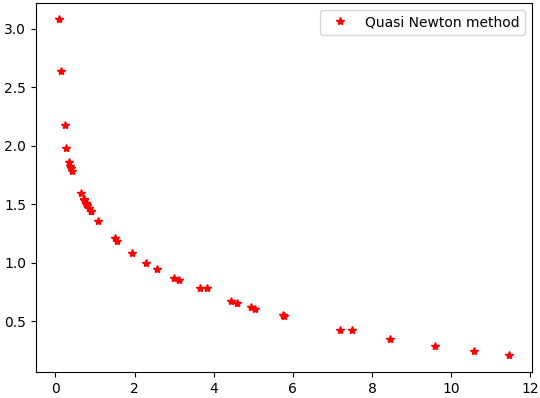}
		\caption{Approximate Pareto front generated by quasi Newton algorithm}
		%\label{fig:three sin x}
	\end{subfigure}
	\hfill
	\begin{subfigure}[b]{0.45\textwidth}
		\centering
		\includegraphics[width=\textwidth]{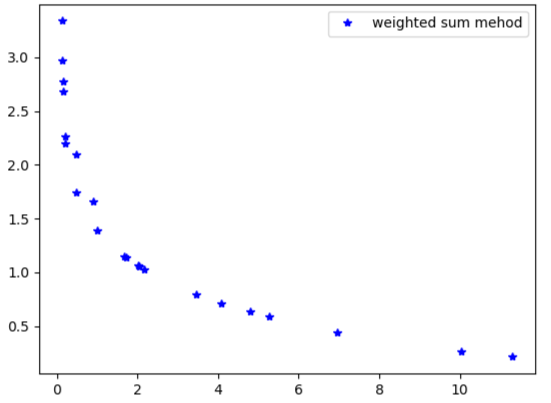}
		\caption{Approximate Pareto front generated by weighted sum method}
		\label{fig:five over x}
	\end{subfigure}
	\caption{Comarison of aproximate Pareto fronts generated by quasi Newton algorithm and weighted sum method.}
	\label{fig1}
\end{figure}
Now we solve $RC^*_1$ with the help of Algorithm \ref{algo1} (quasi Newton method).
In Algorithm \ref{algo1}, we consider the initial point $x^0= (1.10203444,~ 1.93225526),$ $\epsilon=10^{-4},$ and for $j=1,2,$ $i=1,2$ $H_j(x^0,\xi^i)$ is taken as follows
	$$H_1(x^0,\xi^1)=\begin{pmatrix} 0.15716692 & 0.08803005 \\0.08803005 & 0.0844797 \end{pmatrix},~
	H_1(x^0,\xi^2)=\begin{pmatrix} 5.02235556 & 1.7801104  \\1.7801104 & 8.00564985 \end{pmatrix},$$
	$$H_2(x^0,\xi^1)=\begin{pmatrix} 44.14712932 & -9.00388107  \\-9.00388107 & 1.85298502 \end{pmatrix},~                                                                           H_2(x^0,\xi^2)=\begin{pmatrix} 91.45875562 &-26.67973288  \\-26.67973288 &  7.78744484 \end{pmatrix}.$$
	Now, we solve subproblem $Q_N(x)$ at $x^0$ and we get
	 $s^0=(-6.48808533e-03,~9.42476577e-03),$ $t^0= -7.58041271e-06$~ $\alpha_0= 1.$  Then we find $x^1$ such as follows: $x^1=x^0+\alpha_0s^0=(1.09554636,~1.94168003).$ Now we find $F(x^0)$ and $F(x^1)$ as follows  	 $$F_1(x^0)=\max\{\zeta_1(x^0,\xi^1),~\zeta_1(x^0,\xi^2)\},$$                       %$$\zeta_1(x^0,\xi)=(\zeta_1(x^0,\xi^1),~\zeta_1(x^0,\xi^2)),$$
	$$\text{where}~~\zeta_1(x^0,\xi^1)=3.76284389e-05,~\zeta_1(x^0,\xi^2)= 8.12436783e-01.$$
Thus,	$F_1(x^0)=0.81243678.$
Also, $$F_2(x^0)=\max\{\zeta_2(x^0,\xi^1),~\zeta_2(x^0,\xi^2)\},$$
	$$\text{where}~~\zeta_2(x^0,\xi^1)=2.57126451e-01,~\zeta_2(x^0,\xi^2)= 1.49865417.$$
	Thus, $F_2(x^0)=1.49865417.$
	Hence, $F(x^0)=(0.81243678,~ 1.49865417).$

	$$F_1(x^1)=\max\{\zeta_1(x^1,\xi^1),~\zeta_1(x^1,\xi^2)\},$$
	$$\text{where}~~\zeta_1(x^1,\xi^1)=2.66192924e-05,~\zeta_1(x^1,\xi^2)=7.94541995e-01.$$
	Thus,	$F_1(x^1)=0.79454199.$
	$$F_2(x^1)=\max\{\zeta_2(x^1,\xi^1),~\zeta_2(x^1,\xi^2)\}$$
	$$\text{where}~~\zeta_2(x^1,\xi^1)=2.19593161e-01, ~\zeta_2(x^1,\xi^2)=1.49861808.$$
		Thus,	$F_2(x^1)= 1.49861808.$ Hence,
	$F(x^1)=(0.79454199,~ 1.49861808).$ We can observe that $F(x^1)=(0.79454199,~ 1.49861808)<F(x^0)=(0.81243678,~ 1.49865417)$ i.e., the function value decreases component wise in the descent direction. Now we update Hessian matrix with the help of $H_j(x^0,\xi^i).$
	$$H_1(x^1,\xi_1)=\begin{pmatrix} 0.0352717 & 0.02121372 \\0.02121372 & 0.06300125 \end{pmatrix},~
	H_1(x^1,\xi_2)=\begin{pmatrix} 3.18805989 & 1.47432974 \\1.47432974 & 9.75153934 \end{pmatrix},$$
	$$H_2(x^1,\xi_1)=\begin{pmatrix} 43.57305567 & -8.82511574 \\-8.82511574 & 1.80401205\end{pmatrix},~
	H_2(x^1,\xi_2)=\begin{pmatrix} 89.91329472 & -26.32939674\\-26.32939674 &  7.71467964\end{pmatrix}.$$
	With the help of $x^1$, $H_j(x^1,\xi_{i})$ for $j=1,2$ and $i=1,2,$ we solve subproblem and we get $t^1=-5.70455549e-06$, $s^1=(-8.44268112e-{03},~  4.07150091e-03)$ and after 3 iteration we get critical point $x^*=(1.08710368,~
	1.94575153),$ which will be the approximate weak Pareto optimal solution to the problem $RC^*_1$ and the corresponding nondominated point to $x^*$ is given by $F(x^*)=(0.79127969,~ 1.49856185).$ Also, by the Theorem \ref{t1} and Algorithm \ref{algo1}, $x^*$ will be the robust weak Pareto optimal solution to $P(U).$\\
	 \par Now we prove $x^*=(1.08710368,~
	 1.94575153)$ is a weak Pareto optimal solution to the problem $RC^*_1$. By observation we have $\zeta_1(x^*,\xi^1)=1.87211543e-05$, $\zeta_1(x^*,\xi^2)=7.91279695e-01$, $\zeta_2(x^*,\xi^1)=1.82175041e-01$, $\zeta_2(x^*,\xi^2)=1.49856185$ and $F(x^*)=(0.79127969,~ 1.49856185).$ These imply $I_1(x^*)=\{2\}$ and $I_2(x^*)=\{2\}.$ The sub differential of $F_1(x)$ and  $F_2(x)$ at $x^*$ are given as follows:\\
	 $\partial F_1(x^*)=conv\{\nabla\zeta_1(x,\xi^i):i\in I_1(x^*))\}=conv\{(-0.7607892531,~-2.3434674942)\}
	$\\
	  $\partial F_2(x^*)=conv\{\nabla\zeta_2(x,\xi^i):i\in I_2(x^*))\}=conv\{(0.1816117909,~1.38447930717)\}.$ Also
	 $\partial F(x)=
		conv \left( \displaystyle \cup_{j=1}^{2}\partial F_j(x^*) \right)=
	 conv\{(-0.7607892531,-2.3434674942),~(0.1816117909,1.38447930717)\}.$\\
	 Thus $0\in conv \left( \displaystyle \cup_{j=1}^{2}\partial F_j(x^*) \right)$ then $x^*$ will be approximate weak Pareto optimal solution to the problem $RC^*_1$ and the corresponding nondominated point to $x^*$ is given by $F(x^*)=(0.79127969,~ 1.49856185).$
	 Also, by the Theorem \ref{t1}, $x^*$ will be the approximate robust weak Pareto optimal solution to $P(U).$
Comparison of approximate Pareto front generated by  quasi Newton method and weighted sum method of the Example \ref{example1} is given in the Figure \ref{fig1}.
	
\end{example}
\begin{example}\label{example2}\rm\textbf{(Two dimensional three-objective non-convex  problem under uncertainty set of three elements)}\rm\\
	Consider the uncertain three objective optimization problem
	$$P(U)=\{P(\xi):\xi\in U\}$$ such that
	$$P(\xi):~~~~~~ \displaystyle\min_{x\in\mathbb{R}^2} \zeta (x,\xi),$$
	where	$\zeta(x,\xi)=(\zeta_1(x,\xi),\zeta_2(x,\xi),\zeta_3(x,\xi)),$  $\zeta:\mathbb{R}^2 \times U\to \mathbb{R}^3$. Consider $ U=\{(5,3),~(5,6),~(4,1)\}$ such that $\xi^1=(\xi^1_1,\xi^1_2)=(5,3),~\xi^2=(\xi^2_1,\xi^2_2)=(5,6),~\xi^3=(\xi^3_1,\xi^3_2)=(4,1),$ and $\bar{\Lambda}=\{i:\xi_i\in U\}=\{1,2,3\}.$
	\begin{equation*}
		~\zeta(x,\xi^i)=( x_1^2+\xi^i_1x_2^4+\xi^i_1\xi^i_2x_1x_2,~ 5x_1^2+\xi^i_1x_2^2+\xi^i_2x_1^4x_2,~e^{-\xi^i_1x_1+\xi^i_2x_2}+x_1^2-\xi^i_1x_2^2),~\text{where}~ \xi^i=(\xi^i_1,\xi^i_2).
	\end{equation*}
	Therefore,
	\begin{equation*}
		\zeta_1(x,\xi^i)= x_1^2+\xi^i_1x_2^4+\xi^i_1\xi^i_2x_1x_2,~i=1,2,3,
	\end{equation*}
	\begin{equation*}
		\zeta_2(x,\xi^i)=5x_1^2+\xi^i_1x_2^2+\xi^i_2x_1^4x_2,~i=1,2,3,
	\end{equation*}
	\begin{equation*}
		\zeta_3(x,\xi^i)=e^{-\xi^i_1x_1+\xi^i_2x_2}+x_1^2-\xi^i_1x_2^2,~i=1,2,3.
	\end{equation*}
	\begin{figure}[!httb]
		\centering
		%	\begin{subfigure}[b]{0.45\textwidth}
		%	\centering
		%	\includegraphics[width=\textwidth]{example3newton}
		%	\caption{Approximate Pareto front generated by Newton's algorithm for $RC^*_1$}
		%\label{fig:y equals x}
		%	\end{subfigure}
		\hfill
		\begin{subfigure}[b]{0.45\textwidth}
			\centering
			\includegraphics[width=\textwidth]{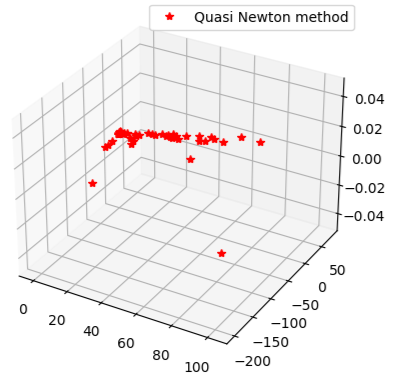}
			\caption{Approximate Pareto front generated by quasi Newton algorithm for $RC^*_1$}
			%\label{fig:three sin x}
		\end{subfigure}
		\hfill
		\begin{subfigure}[b]{0.46\textwidth}
			\centering
			\includegraphics[width=\textwidth]{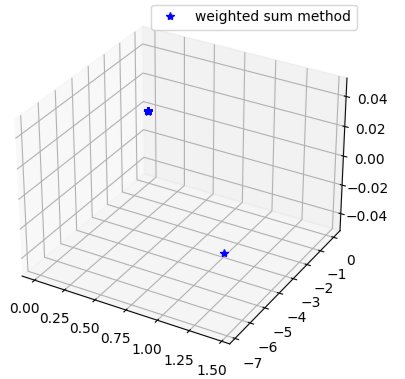}
			\caption{Approximate Pareto front generated by weighted sum method for $RC^*_1$}
			\label{fig:five over x}
		\end{subfigure}
		\caption{Comarison of aproximate Pareto fronts generated by quasi Newton algorithm and weighted sum method for Example \ref{example2}}
		\label{fig4}
	\end{figure}
	
	Objective wise worst case cost type robust counterpart to $P(U)$ is given by
	\begin{equation}
		RC^*_1:~~~~~	\displaystyle\min_{x\in \mathbb{R}^2}F(x),\end{equation}
	\begin{equation*}
		\text{where}~	F(x)=(F_1(x), F_2(x))~\text{and}
	\end{equation*}
	\begin{equation*}
		F_1(x)=\displaystyle \max_{i\in \bar{\Lambda}} \zeta_1(x,\xi^i)=\displaystyle \max_{i\in \bar{\Lambda}}\{x_1^2+5x_2^4+15x_1x_2,~x_1^2+5x_2^4+30x_1x_2,~x_1^2+4x_2^4+4x_1x_2\},
	\end{equation*}
	\begin{equation*}F_2(x)=\displaystyle \max_{i\in \bar{\Lambda}} \zeta_2(x,\xi^i)=\displaystyle \max\{5x_1^2+5x_2^2+3x_1^4x_2,~5x_1^2+5x_2^2+6x_1^4x_2,~ 5x_1^2+4x_2^2+x_1^4x_2\}
	\end{equation*}
	\begin{equation*}F_3(x)=\displaystyle \max_{i\in \bar{\Lambda}} \zeta_3(x,\xi^i)=\displaystyle \max\{e^{-5x_1+3x_2}+x_1^2-5x_2^2,~e^{-5x_1+6x_2}+x_1^2-5x_2^2,~e^{-4x_1+x_2}+x_1^2-4x_2^2\}.
	\end{equation*}
%	Comparison of approximate Pareto fronts generated by quasi Newton method and weighted sum method of the Example \ref{example5} is given in the Figure \ref{fig4}.
	\end{example}
%%%%%%%%%%%%%%%%%
\begin{example}\label{example3}\textbf{(Two dimensional three objective convex problem under uncertainty set of three elements)}\rm\\
	Consider the uncertain bi-objective optimization problem
	$$P(U)=\{P(\xi):\xi\in U\}$$ such that
	$$P(\xi):~~~~~~ \displaystyle\min_{x\in\mathbb{R}^2} \zeta (x,\xi),$$
	where	$\zeta(x,\xi)=(\zeta_1(x,\xi),\zeta_2(x,\xi)),$  $\zeta:\mathbb{R}^2 \times U\to \mathbb{R}^3.$ Consider $ U=\{(4,1),~(5,2),~(6,4)\}$ such that $\xi^1=(\xi^1_1,\xi^1_2)=(5,3),~\xi^2=(\xi^2_1,\xi^2_2)=(5,6),~\xi^3=(\xi^3_1,\xi^3_2)=(4,1),$ and $\bar{\Lambda}=\{i:\xi_i\in U\}=\{1,2,3\}.$
	\begin{equation*}
		\zeta(x,\xi^i)=(100\xi_1 (x_2-x_1^2)^2+\xi^i_2(1-x_1)^2,(x_2-\xi^i_1)^2+\xi^i_2x_1^2,~\xi^i_1x_1^2+3\xi^i_2x_2^2 ),
	\end{equation*}
	where $\xi=(\xi^i_1,\xi^i_2)$.
	\begin{equation*}
		\zeta_1(x,\xi^i)=100\xi^i_1 (x_2-x_1^2)^2+\xi^i_2(1-x_1)^2,~i=1,2,3,
	\end{equation*}
	\begin{equation*}
		\zeta_2(x,\xi^i)=(x_2-\xi^i_1)^2+\xi^i_2x_1^2,~i=1,2,3,
	\end{equation*}
	\begin{equation*}
		\zeta_3(x,\xi^i)=\xi^i_1x_1^2+3\xi^i_2x_2^2~i=1,2,3.
	\end{equation*}
	\begin{figure}[!httb]
	\centering
	%\begin{subfigure}[b]{0.45\textwidth}
	%	\centering
	%	\includegraphics[width=\textwidth]{example2newton}
	%	\caption{Approximate Pareto front generated by Newton's algorithm for $RC^*_1$}
	%\label{fig:y equals x}
	%\end{subfigure}	
	\begin{subfigure}[b]{0.45\textwidth}
		\centering
		\includegraphics[width=\textwidth]{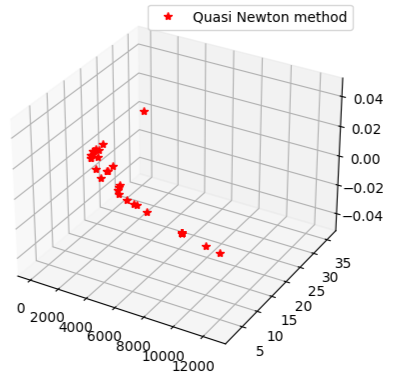}
		\caption{Approximate Pareto front generated by quasi Newton algorithm for $RC^*_1$}
		%\label{fig:three sin x}
	\end{subfigure}
	\hfill
	\begin{subfigure}[b]{0.45\textwidth}
		\centering
		\includegraphics[width=\textwidth]{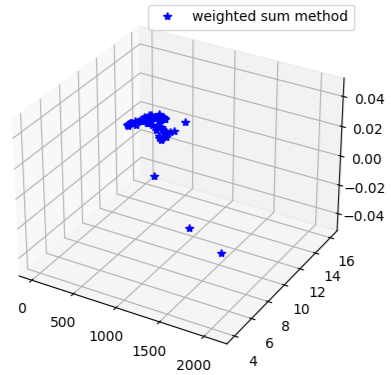}
		\caption{Approximate Pareto front generated by weighted sum method for $RC^*_1$}
		%\label{fig:five over x}
	\end{subfigure}
	\caption{Comarison of aproximate Pareto fronts generated by quasi Newton algorithm and weighted sum method for Example \ref{example3}}
	\label{fig3}
\end{figure}
	Objective wise worst case cost type robust counterpart to $P(U)$ is given by
	\begin{equation}
			RC^*_1:~~~~~	\displaystyle\min_{x\in \mathbb{R}^2}F(x),\end{equation}
	\begin{equation*}
		\text{where}~F(x)=(F_1(x), F_2(x),  F_3(x) )~\text{and}
	\end{equation*}	
	\begin{equation*}
		F_1(x)=\displaystyle \max_{i\in \bar{\Lambda}} \zeta_1(x,\xi^i)=\displaystyle \max\{400 (x_2-x_1^2)^2+(1-x_1)^2,~500 (x_2-x_1^2)^2+2(1-x_1)^2,~~600 (x_2-x_1^2)^2+4(1-x_1)^2\},
	\end{equation*}
	\begin{equation*}F_2(x)=\displaystyle \max_{i\in \bar{\Lambda}}\zeta_2(x,\xi^i)=\displaystyle \max\{(x_2-4)^2+x_1^2,~(x_2-5)^2+2x_1^2,~~(x_2-6)^2+4x_1^2\}
	\end{equation*}
	\begin{equation*}
		F_3(x)=\displaystyle\max_{i\in \bar{\Lambda}} \zeta_3(x,\xi^i)=\displaystyle \max\{4x_1^2+3x_2^2,~5x_1^2+6x_2^2,~6x_1^2+12x_2^2\}.
	\end{equation*}
%	Comparison of approximate Pareto fronts generated by quasi Newton method and weighted sum method of the Example \ref{example4} are given in the Figure \ref{fig3}.
\end{example}
Based on the examples discussed above, we can observe that the objective functions in Example \ref{example1} and Example \ref{example2} are nonconvex, while those in Example \ref{example3} are non convex. From the figures associated with each example, it is evident that Algorithm \ref{algo1} successfully generates good approximations of the Pareto fronts in both convex and non convex cases. However, when examining the Figure of Example \ref{example2}, it becomes clear that the weighted sum method fails to produce even an approximate front in the non convex case.
\par In addition to these examples 20 test problems are constructed and the weighted sum method and Algorithm \ref{algo1} are compared using performance profiles.  \\
{\bf Test Problems:} The above two are included in a set of test problems that are generated and executed using both techniques. Table \ref{table1} provides details of these test problems. Table \ref{table1}'s columns $m$ and $n$ stand for the number of objective functions and the size of the decision variable ($x$). One can see that in our execution, we took into account 2 or 3 objective $1,$ $2,$ $3,$  $5,$ and $10$ dimensional problems. Each test problem's explicit form is provided in Appendix A.
\begin{rem}
	Since we developed a quasi Newton method for $RC^*_1$ which is the robust counterpart of $P(U).$ Therefore, $RC^*_1$ is considered as the test problem corresponding to some $P(U)$ in the Appendix A.
	\end{rem}
\begin{table}[!ht]
	%\scriptsize
	\begin{center}
		\begin{tabular}
			{|c|c|c|c|}\hline
			Problem & $(m,n,p)$& $lb^T$ & $ub^T$\\ \hline
			
			%TP1 & (2,2,2)& $(-4,-4)^T$ & $(4,4)^T$ \\ \hline
		TP1 &	$(2,2,2)$&$(-2,-2)^T$&$(5,5)^T$ \\ \hline
		%	TP2 & (3,2,3)& $(-1,-1)^T$ & $(0,0)^T$\\ \hline
			TP2 &$(3,2,3)$&$(-1,-1)^T$&$(5,2)^T$\\ \hline
		%	TP3 & (3,2,3)& $(-1,-1)^T$ & $(5,2)^T$ \\ \hline
		TP3 &$(3,2,3)$&$(-1,-1)^T$&$(5,2)^T$\\ \hline
		%	TP4 & (3,3,3)& $(1,-2,0)^T$ & $(3.5,2,1)^T$ \\ \hline
			TP4& $(2,2,2)$&$(-5,-5)^T$&$(5,5)^T$\\ \hline
		%	TP5 & (2,2,2)& $(-6,-6)^T$ & $(6,4)^T$ \\ \hline
			TP5 &$(2,1,2)$&-5&5\\ \hline
			TP6&$(2,3,3)$&$(0,0,0)^T$&$(1,1,1)^T$\\ \hline
			TP7&$(2,1,2)$&-3&3\\ \hline
			TP8&$(2,2,2)$&$(-4,-4)$&$(4,4)$\\ \hline
			TP9&$(2,3,3)$&$(-1,-2,-1)^T$&$(1,1,2)^T$\\ \hline
			TP10&(3,3,3)&$(1,-2,0)^T$&$(3.5,2.0,1.0)^T$\\ \hline
			TP11&(2,2,2)&$(-6,-6)^T$&$(6,4)^T$\\ \hline
			TP12&(3,3,3)&$(-1,-1,-1)$&$(5,5,5)$\\ \hline
			TP13&(3,2,3)&$(-1,-1,-1)$&$(5,5,5)$\\ \hline
			TP14&(2,1,2)&$-100$&$100$\\ \hline
			TP15&(2,2,2)&$(-2,-2)$&$(5,5)$\\ \hline
			TP16&(3,3,3)&$(0,0,0)$&$(1,1,1)$\\ \hline
			TP17&(3,2,3)&$(-4,-4)$&$(5,5)$\\ \hline
			TP18&(2,2,2)&$(.01,.001)$&$(1,1)$\\ \hline
			TP19&(2,5,2)&$(.001,...,001)$&$(1,...,1)$\\ \hline
			TP20&(3,10,3)&$(0.001,..,0.001)$&$(1,...,1)$\\ \hline
		%	TP6 & (2,1,2)& -5 &5\\ \hline
			%TP6 & (3,3,3)& $(-1,-1,-1)^T$ & $(5,5,5)^T$ \\ \hline
		%	TP7 & (3,3,3)& $(-1,-1,-1)^T$ & $(5,5,5)^T$ \\ \hline
		%	TP8 & (3,2,3)& $(-1,-1)^T$ & $(5,2)^T$ \\ \hline
			%TP9 & (2,2,2)& $(-1,-1)^T$ & $(0,0)^T$ \\ \hline
		%	TP9 & (3,2,3)& $(-1,-1)^T$ & $(0,0)^T$ \\ \hline
		%	TP10 & (2,2,2)& $(-2,-2)^T$ & $(5,5)^T$ \\ \hline
		%		TP11 & (2,2,2)& $(-2,-2)^T$ & $(5,5)^T$ \\ \hline
		%	TP12 & (2,1,2)& $-100$ & $100$ \\ \hline
		%	TP13 & (2,2,2)& $(-2,-2)^T$ & $(5,5)^T$ \\ \hline
		%	TP14 & (3,2,3)& $(0,0)^T$ & $(1,1)^T$ \\ \hline
		%	TP15 & (3,3,3)& $(0,0,0)^T$ & $(1,1,1)^T$ \\ \hline
		%	TP16 & (2,5,3)& $(0.001,-1,-1,-1,-1)^T$ & $(1,1,\dots,1)^T$ \\ \hline
		%	TP17 & (2,2,2)& $(0,0.001)^T$ & $(1,1)^T$ \\ \hline
		%	TP18a & (2,2,2)& $(0.001,0.001)^T$ & $(1,1)^T$ \\ \hline
		%	TP18b & (2,2,3)&$(0.001,0.001)^T$ & $(1,1)^T$ \\ \hline
		%	TP18b& (3,10,3)& $(0.001,\ldots,0.001)^T$ & $(1,\ldots,1)^T$\\ \hline
			%TP19& (3,10,3)& $(0.001,\ldots,0.001)^T$ & $(1,\ldots,1)^T$\\ \hline
		\end{tabular}
		\caption{Details of test problems}
		\label{table1}
	\end{center}
\end{table}
%\begin{table}[!ht]
%	%\scriptsize
%	\begin{center}
%		\begin{tabular}
%			{|c|c|c|c|}\hline
%			Problem & $(m,n,p)$& $lb^T$ & $ub^T$\\ \hline
%			
%			TP1 & (2,2,2)& $(-4,-4)^T$ & $(4,4)^T$ \\ \hline
%			TP2 & (2,1,2)& -5 &5 \\ \hline
%			TP3 & (2,3,3)& $(0,0)^T$ & $(1,1)^T$ \\ \hline
%			TP4 & (3,3,3)& $(1,-2,0)^T$ & $(3.5,2,1)^T$ \\ \hline
%			TP5 & (2,2,2)& $(-6,-6)^T$ & $(6,4)^T$ \\ \hline
%			TP6 & (3,3,3)& $(-1,-1,-1)^T$ & $(5,5,5)^T$ \\ \hline
%			TP7 & (3,3,3)& $(-1,-1,-1)^T$ & $(5,5,5)^T$ \\ \hline
%			TP8 & (3,2,3)& $(-1,-1)^T$ & $(5,2)^T$ \\ \hline
%			TP9 & (2,2,2)& $(-1,-1)^T$ & $(0,0)^T$ \\ \hline
%			TP10 & (2,2,2)& $(-2,-2)^T$ & $(5,5)^T$ \\ \hline
%			TP11 & (2,1,2)& $-100$ & $100$ \\ \hline
%			TP12 & (2,2,2)& $(-2,-2)^T$ & $(5,5)^T$ \\ \hline
%			TP13 & (3,2,3)& $(0,0)^T$ & $(1,1)^T$ \\ \hline
%			TP14 & (3,3,3)& $(0,0,0)^T$ & $(1,1,1)^T$ \\ \hline
%			TP15 & (2,5,3)& $(0.001,-1,-1,-1,-1)^T$ & $(1,1,\dots,1)^T$ \\ \hline
%			TP16 & (2,2,2)& $(0,0.001)^T$ & $(1,1)^T$ \\ \hline
%			TP17 & (2,2,2)& $(0.001,0.001)^T$ & $(1,1)^T$ \\ \hline
%			TP18a & (2,2,3)&$(0.001,0.001)^T$ & $(1,1)^T$ \\ \hline
%			TP18b& (3,10,3)& $(0.001,\ldots,0.001)^T$ & $(1,\ldots,1)^T$\\ \hline
%			TP19& (3,10,3)& $(0.001,\ldots,0.001)^T$ & $(1,\ldots,1)^T$\\ \hline
%		\end{tabular}
%		\caption{Details of test problems}
%		\label{table1}
%	\end{center}
%\end{table}
\par The weighted sum method and Algorithm \ref{algo1} are compared using performance profiles.~Performance profiles are employed to contrast various approaches (see \citet{ansary2015modified,ansary2019sequential,ansary2020sequential} for more details of performance profiles). The cumulative function $\rho(\tau)$ that represents the performance ratio in relation to a specific metric and a set of methods is known as a performance profile. Give a set of methods $\mathcal{SO}$ and a set of problems $\mathcal{P}$, let $\varsigma_{p,s}$ be the performance of solver $s$ on solving $p$. The performance ratio is defined as $r_{p,s}=\varsigma_{p,s}/\min_{s\in\mathcal{SO}}\varsigma_{p,s}$. The cumulative function $ \rho_s(\tau)~~(s\in\mathcal{SO})$ is defined as
\begin{eqnarray*}
	\rho_s(\tau)=\frac{|\{p \in \mathcal{P}:r_{p,s}\leq \tau\}|}{|P|}.
\end{eqnarray*}
To justify  how much well-distributed this set is, the following metrics are considered for computing performance profile.\\\\
\textbf{$\Delta$-spread metric:} Let $x^1,x^2,\dots,x^N$ be the set of points obtained by a solver $s$ for problem $p$ and let these points be sorted by $f_j(x^i)\leq f_j(x^{i+1})$ $(i=1,2,\dots,N-1)$.
%Suppose $x^0$ is the best known approximation of global minimum of $f_j$ and $x^{N+1}$ is the best known global maximum of $f_j$, computed over all the approximated Pareto fronts obtained by different solvers.
Assume that $x^0$, calculated over all the approximated Pareto fronts obtained by various solvers, is the best known approximation of the global minimum of $f_j$, and $x^{N+1}$, the best known approximation of the global maximum of $f_j$.
Define $\bar{\delta_j}$ as the average of the distances $\delta_{i,j}$, $ i=1,2,\dots,N-1.$ For an algorithm  $s$ and a problem $p$, the spread metric $\Delta_{p,s}$ is
\begin{eqnarray*}
	\Delta_{p,s}:=\underset{j\in\Lambda_m}{\max}\left(\frac{\delta_{0,j}+\delta_{N,j}+\Sigma_{i=1}^{N-1} |\delta_{i,j}-\bar{\delta_j}|}{\delta_{0,j}+\delta_{N,j}+(N-1)\bar{\delta_j}}\right).
\end{eqnarray*}
\textbf{Hypervolume metric:} Hypervolume metric of an approximate Pareto front with respect to a reference point $P_{ref}$ is defined as the volume of the entire region dominated by the efficient solutions obtained by a method with respect to the reference point. We have used the codes from \url{https://github.com/anyoptimization/pymoo} to calculate hypervolume metric. Higher values of $hv_{p,s}$ indicate better performance using hypervolume metric. So while using the performance profile of the solvers measured by hypervolume metric we need to set $\widetilde{hv_{p,s}}=\frac{1}{hv_{p,s}}$.\\
Performance profile using $\Delta$ spread metric and hypervolume metric are given in Figures \ref{figuredelta} and \ref{figurehypervolume} respectively.
\begin{figure}[!htbp]
	\centering
	\includegraphics[width=6.0in,height=3.0in]{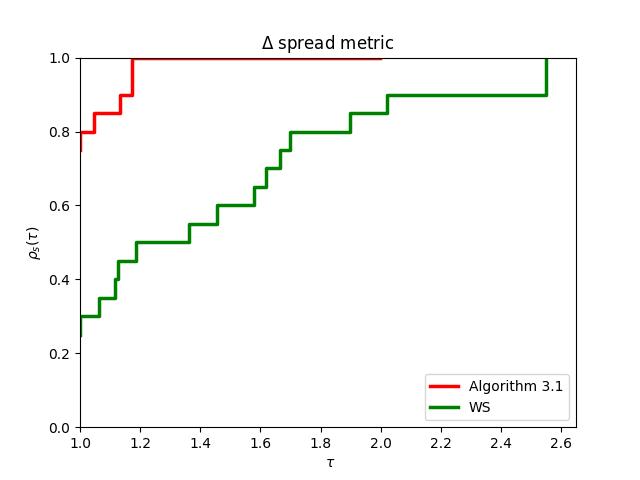}
	\caption{Performance profile using $\Delta$ spread metric }
	\label{figuredelta}
\end{figure}
\begin{figure}[!htbp]
	\centering
	\includegraphics[width=6.50in,height=3.0in]{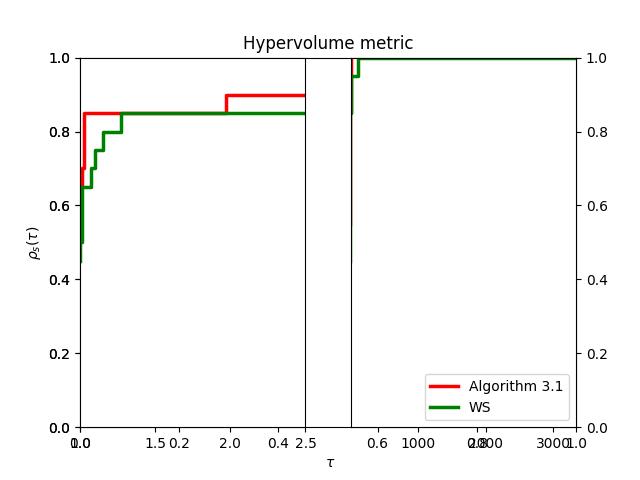}
	\caption{Performance profile using Hypervolume metric }
	\label{figurehypervolume}
\end{figure}
\\
In addition, using the number of iterations and function evaluations, we computed performance profiles. Both methods use the forward difference formula to calculate gradients, which necessitates $n$ additional function evaluations.~If {\it \# Iter} and {\it \# Fun} denote the number of iterations and number of function evaluations required to solve a problem respectively then total function evaluation is ${\it \# Fun}+n  {\it \# Iter}$. Figures \ref{figurenumberofiteration} and \ref{figurefunctionevaluation} show the performance profiles using the number of iterations and the number of function evaluations, respectively.
\begin{figure}[!htbp]
	\centering
	\includegraphics[width=6.50in,height=3.0in]{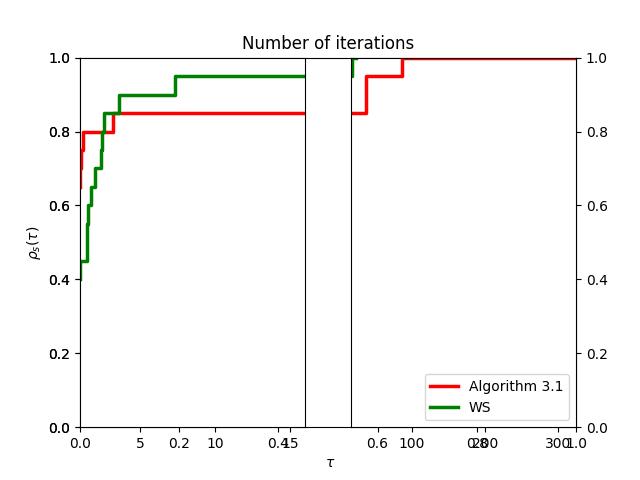}
	\caption{Performance profile using number of iterations}
	\label{figurenumberofiteration}
\end{figure}
\begin{figure}[!htbp]
	\centering
	\includegraphics[width=6.5in,height=3.0in]{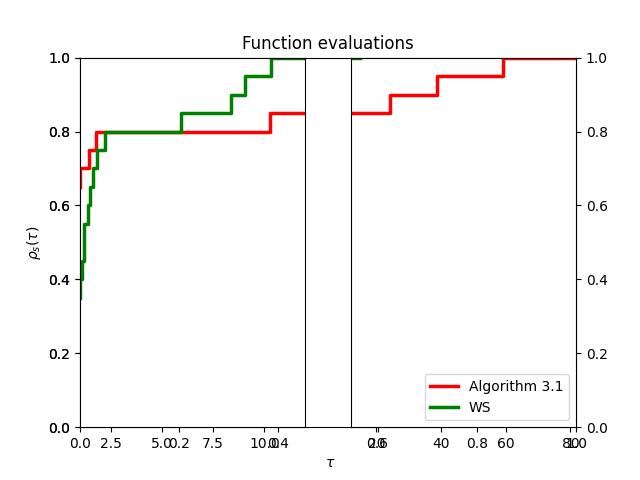}
	\caption{Performance profile using function evaluations}
	\label{figurefunctionevaluation}
\end{figure}
The performance profile figures show that the Algorithm \ref{algo1} typically outperforms the weighted sum method.
%One can observe from performance profile figures that Algorithm \ref{algo1} performs better than weighted sum method in most cases.
%%%%%%%%%%%%%%%%%%%%%%%%%%%%%%%%%%%%%%%%%%
%%%%%%%%%%%%%%%%%%%%%%%%%%%%%%
%Based on the examples discussed above, we can observe that the objective functions in Example \ref{example1} and Example \ref{example4} are convex, while those in Example \ref{example1}, Example \ref{example2}, Example \ref{example5}, and Example \ref{example6} are non convex. From the figures associated with each example, it is evident that Algorithm \ref{algo1} successfully generates good approximations of the Pareto fronts in both convex and non  convex cases. However, when examining the figures of Example \ref{example2}, Example \ref{example5}, and Example \ref{example6}, it becomes clear that the weighted sum method fails to produce even an approximate front in the non convex case. Although the objective function in Example \ref{example1} is also non convex, the weighted sum method manages to generate an approximate Pareto front, but inferior to the one produced by Algorithm \ref{algo1}.
 \par The advantage of Algorithm \ref{algo1} is that it eliminates the need to choose pre-order weights, which is a challenge in the weighted sum method. Consequently, Algorithm \ref{algo1} enables us to obtain good approximations of Pareto fronts without the requirement of pre-order weight selection. Theorem \ref{t1} demonstrate that the Pareto optimal solution for the problem $RC^*_1$ is the robust Pareto optimal solution to the $P(U).$ Consequently, the approximate Pareto front obtained for $RC^*_1$ is also an approximate Pareto front for $P(U)$. By utilizing $RC^*_1$ and Algorithm \ref{algo1}, our main objective of discovering the approximate Pareto front for the multiobjective optimization problem $P(U)$ with a finite parameter uncertainty set has been achieved.
%%%%%%%%%%%%%%%%%%%%%%%%%%%%%%%%%
\section{Conclusions}\label{s4}
We solved an uncertain multiobjective optimization problem $P(U)$ defined in (\ref{mp}), using its robust counterpart $RC^*_1$, which is a deterministic multiobjective optimization problem. By solving $RC^*_1$, we obtain the solution for $P(U)$, eliminating the need to directly solve $P(U)$. To tackle $RC^*_1$, we employed a quasi Newton method, assuming strong convexity of the objective function. This assumption ensures that the necessary condition for optimality is also sufficient for Pareto optimality. Our approach involved developing a quasi Newton algorithm to find critical points. This algorithm solves a subproblem to determine the quasi Newton descent direction and employs an Armijo-type inexact line search technique. We also incorporated a modified BFGS formula. By combining these elements, we constructed a quasi Newton algorithm that generates a sequence converging to a critical point. This critical point serves as the Pareto optimal solution for $RC^*_1$ and the robust Pareto optimal solution for the uncertain multiobjective optimization problem $P(U)$. Under certain assumptions, we proved that the sequence generated by the quasi Newton algorithm converges to a critical point with a superlinear rate. Finally, we validated the algorithm through numerical examples. We compared the approximate Pareto front obtained by the quasi Newton algorithm with the one generated by the weighted sum method. We observed that Algorithm \ref{algo1} (quasi Newton method) successfully generates good approximations of the Pareto fronts in both convex and non convex cases. However, the weighted sum method fails to produce even an approximate front in the non convex case. In addition to this, we have computed performance profiles using number of iterations, function evaluations, $\Delta$ spread metric and hypervolume metric. On the basis of performance profile figures that quasi Newton method performs better than weighted sum method in most cases.
\par In this work, we were focused on finding the solution to an uncertain multiobjective optimization problem under a finite uncertainty set. The solution of an uncertain multiobjective optimization problem under an infinite uncertainty set is left for future scope.
\section*{Acknowledgment:}
This research is supported by Govt. of India CSIR fellowship, Program No. 09/1174(0006)/2019-EMR-I, New Delhi.
\section*{Conflict of interest:}
The authors declared that they have no conflict of interests.
% % % % % % % % % % % % % % % % % % % % % % % % % % % % % % % % % % % % % % % % % % % % %
\bibliographystyle{apa}
\bibliography{quasi}

\newpage
\appendix
\section{Details of test problems}
\begin{appendices}{\bf Appendix A: Details of test problems}
	\begin{enumerate}[({TP}1)]
			\item $F: \mathbb{R}^2\rightarrow \mathbb{R}^2$ is defined as $F_j(x)= \max\{\zeta_{j}(x,\xi^{i}): i=1,2\},$ $j=1,2,$ where $\xi^i=(\xi^i_1,\xi^i_2)$ and $\xi^1=(1,2),$ $\xi^2=(2,3).$
		\begin{eqnarray*}
				\zeta_1(x,\xi^i)=\frac{1}{4}(x_1-\xi^i_1)^4+2(x_2-\xi^i_2)^4~~i=1,2\\\mbox{and}~~~~	\zeta_2(x,\xi^i)=(\xi^i_1x_2-\xi^i_2x_1^2)^2+(1-\xi^i_1x_1)^2,~i=1,~2.
		\end{eqnarray*}
	%m=2,n=2,p=2,lb=(-2,-2)^T,ub=(5,5)^T
		\item $F: \mathbb{R}^2\rightarrow \mathbb{R}^3$ is defined as $F_j(x)= \max\{\zeta_j(x,\xi^{i}): i=1,2,3\},$ $j=1,2,3$ where $\xi^i=(\xi^i_1,\xi^i_2)$ and $\xi^1=(5,3),$ $\xi^2=(5,6),$ $\xi^3=(4,1).$
			\begin{equation*}
			\zeta_1(x,\xi^i)= x_1^2+\xi^i_1x_2^4+\xi^i_1\xi^i_2x_1x_2,
		\end{equation*}
		\begin{equation*}
			\zeta_2(x,\xi^i)=5x_1^2+\xi^i_1x_2^2+\xi^i_2x_1^4x_2,
		\end{equation*}
		\begin{equation*}
			\zeta_3(x,\xi^i)=e^{-\xi^i_1x_1+\xi^i_2x_2}+x_1^2-\xi^i_1x_2^2.
		\end{equation*}
%m=3,n=2,p=3,lb=(-1,-1)^T,ub=(5,2)^T.
	\item  $F: \mathbb{R}^2\rightarrow \mathbb{R}^3$ is defined as $F_j(x)= \max\{\zeta_j(x,\xi^{i}): i=1,2,3\},$ $j=1,2,3$ where $\xi^i=(\xi^i_1,\xi^i_2)$ and $\xi^1=(4,1),$ $\xi^2=(5,2),$ $\xi^3=(6,4).$
\begin{equation*}
	\zeta_1(x,\xi^i)=100\xi^i_1 (x_2-x_1^2)^2+\xi^i_2(1-x_1)^2,
\end{equation*}
\begin{equation*}
	\zeta_2(x,\xi^i)=(x_2-\xi^i_1)^2+\xi^i_2x_1^2,
\end{equation*}
\begin{equation*}
	\zeta_3(x,\xi^i)=\xi^i_1x_1^2+3\xi^i_2x_2^2.
\end{equation*}
%m=3,n=2,p=3,lb=(-1,-1)^T,ub=(5,2)^T
		\item $F: \mathbb{R}^2\rightarrow \mathbb{R}^2$ is defined as $F_j(x)= \max\{\zeta_j(x,\xi^{i}): i=1,2\}$ $j=1,2,$ where $\xi^1=(1,3)^T,~ ~\xi^2=(3,1)^T.$
		Note that
		$\zeta_{1}(x,\xi^1)=(x_1-1)^2+(x_2-3)^2$, and $\zeta_{1}(x,\xi^2)=(x_1-3)^2+(x_2-1)^2$, $\zeta_{2}(x,\xi^1)=x_1^2+3x_2^2$, and $\zeta_{2}(x,\xi^{2})=3x_1^2+x_2^2$.
		%m=2,n=2,p=2,lb=(-5,-5)^T,ub=(5,5)^T
	\item  $F: \mathbb{R}\rightarrow \mathbb{R}^2$ is defined as $F_j(x)= \max\{\zeta_j(x,\xi^{i}): i=1,2\}$ $j=1,2,$ where $\xi^1=-1$, $\xi^2=3$, $\zeta_1(x,\xi^i)=(x-\xi^i)^2$ and $\zeta_2(x,\xi^i)=x^2+\xi^i x$ for $i=1,2$.
		%m=2,n=1,p=2,lb=-5,ub=5
\item  $F: \mathbb{R}^3\rightarrow \mathbb{R}^2$ is defined as $F_j(x)= \max\{\zeta_j(x,\xi^{i}): i=1,2,3\},$ $j=1,2.$ Here $\xi^i=(\xi^i_1,\xi^i_2,\xi^i_3)$ and therefore $\xi^1=(1,1,1)^T,$  $\xi^2=(1,-1,1)^T$, $\xi^3=(1,-2,2)^T$,
		\begin{eqnarray*}
			\zeta_1(x,\xi^i)=&1-e^{-\sum_{j=1}^3\xi^j_{i} \left(x_j-\frac{1}{\sqrt{3}}\right)^2}~~i=1,2,3\\ \mbox{and}~~~~
			\zeta_2(x,\xi^i)=&1-e^{-\sum_{j=1}^3\xi^j_{i} \left(x_j+\frac{1}{\sqrt{3}}\right)^2}~~i=1,2,3.
		\end{eqnarray*}
%m=2,n=3,p=3,lb=(0,0,0)^T,ub=(1,1,1)^T
	\item $F: \mathbb{R}\rightarrow \mathbb{R}^2$ is defined as $F_j(x)= \max\{\zeta_j(x,\xi^{i}): i=1,2\}$ $j=1,2,$ where $\xi^1=-4$, $\xi^2=7$, $\zeta_1(x,\xi^i)=(x-\xi^i)^2$ and $\zeta_2(x,\xi^i)=x^2+\xi^i x$ for $i=1,2$.
	%m=2,n=3,p=3,lb=(0,0,0)^T,ub=(1,1,1)^T
	%m=2,n=1,p=2,lb=-3,ub=3
	\item $F: \mathbb{R}^2\rightarrow \mathbb{R}^2$ is defined as $F_j(x)= \max\{\zeta_j(x,\xi^{i}): i=1,2\}$ $j=1,2,$ where $\xi^1=(1,3)^T,~ ~\xi^2=(3,1)^T.$
	Note that
	$\zeta_{1}(x,\xi^1)=(x_1-1)^2+(x_2-3)^2$, and $\zeta_{1}(x,\xi^2)=(x_1-3)^2+(x_2-1)^2$, $\zeta_{2}(x,\xi^1)=x_1^2+3x_2^2$, and $\zeta_{2}(x,\xi^{2})=3x_1^2+x_2^2$.
	%m=2,n=3,p=3,lb=(0,0,0)^T,ub=(1,1,1)^T
	%m=2,n=1,p=2,lb=-3,ub=3
	%m=2,n=2,p=2,lb=(-4,-4),ub=(4,4)
		\item $F: \mathbb{R}^3\rightarrow \mathbb{R}^2$ is defined as $F_j(x)= \max\{\zeta_j(x,\xi^{i}): i=1,2,3\},$ $j=1,2.$ Here $\xi^i=(\xi^i_1,\xi^i_2,\xi^i_3)$ and therefore $\xi^1=(1,1,1)^T,$  $\xi^2=(1,-1,1)^T$, $\xi^3=(1,-2,2)^T$,
	\begin{eqnarray*}
		\zeta_1(x,\xi^i)=&1-e^{-\sum_{j=1}^3\xi^j_{i} \left(x_j-\frac{1}{\sqrt{3}}\right)^2}~~i=1,2,3\\ \mbox{and}~~~~
		\zeta_2(x,\xi^i)=&1-e^{-\sum_{j=1}^3\xi^j_{i} \left(x_j+\frac{1}{\sqrt{3}}\right)^2}~~i=1,2,3.
	\end{eqnarray*}
%m=2,n=3,p=3,lb=(0,0,0)^T,ub=(1,1,1)^T
%m=2,n=1,p=2,lb=-3,ub=3
%m=2,n=2,p=2,lb=(-4,-4),ub=(4,4)
%m=2,n=3,p=3,lb=(-1,-2,-1)^T,ub=(1,1,2)^T
\item $F: \mathbb{R}^3\rightarrow \mathbb{R}^3$ is defined as $F_j(x)= \max\{\zeta_j(x,\xi^{i}): i=1,2,3\}$ $j=1,2,3,$ where $\xi^i=(\xi^i_1,\xi^i_2,\xi^i_3)$ and
$\xi^1=(1,1,1)^T$, $\xi^2=(1,-1,1)^T$, $\xi^3=(1,-2,2)^T.$
\begin{eqnarray*}
	\zeta_1(x,\xi^i)&=&\left(1+\xi^i_3x_3\right)\left(\xi^i_{1}\xi^i_{2} x_1^3x_2^3-10\xi^i_{1}x_1-4\xi^i_{2}x_2\right)~~i=1,2,3\\
	\zeta_2(x,\xi^i)&=&\left(1+\xi^i_3x_3\right)\left(\xi^i_{1}\xi^i_{2}x_1^3x_2^3-10\xi^i_{1}x_1+4\xi^i_{2}x_2\right)~~i=1,2,3\\
	\mbox{and}~~~~
	\zeta_3(x,\xi^i)&=&\left(1+\xi^i_{3}x_3\right)\xi^i_{1}x_1^2~~i=1,2,3.
\end{eqnarray*}
%m=2,n=3,p=3,lb=(0,0,0)^T,ub=(1,1,1)^T
%m=2,n=1,p=2,lb=-3,ub=3
%m=2,n=2,p=2,lb=(-4,-4),ub=(4,4)
%m=2,n=3,p=3,lb=(-1,-2,-1)^T,ub=(1,1,2)^T
%m=3,n=3,p=3,lb=(1,-2,0)^T,ub=(3.5,2.0,1.0)^T
\item  $F: \mathbb{R}^2\rightarrow \mathbb{R}^2$ is defined as $F_j(x)= \max\{\zeta_j(x,\xi^{i}): i=1,2\},$ $j=1,2,$ where $\xi^i=(\xi^i_1,\xi^i_2)$ and $\xi^1=\{(2,2)$, $\xi^2=(0,4).$
\begin{eqnarray*}
	\zeta_1(x,\xi^i)=(x_1-\xi^i_1)^2+(x_2+\xi^i_2)^2~~i=1,2\\ \mbox{and}~~~~
	\zeta_2(x,\xi^i)=(\xi^i_1x_1+\xi^i_2x_2)^2~~i=1,2.
\end{eqnarray*}
%m=2,n=3,p=3,lb=(0,0,0)^T,ub=(1,1,1)^T
%m=2,n=1,p=2,lb=-3,ub=3
%m=2,n=2,p=2,lb=(-4,-4),ub=(4,4)
%m=2,n=3,p=3,lb=(-1,-2,-1)^T,ub=(1,1,2)^T
%m=3,n=3,p=3,lb=(1,-2,0)^T,ub=(3.5,2.0,1.0)^T
%m=2,n=2,p=2,lb=(-6,-6)^T,ub=(6,4)^T.
\item $F: \mathbb{R}^3\rightarrow \mathbb{R}^3$ is defined as $F_j(x)= \max\{\zeta_j(x,\xi^{i}): i=1,2,3\},$ $j=1,2,3,$ where $\xi^i=(\xi^i_1,\xi^i_2)$ and $\xi^1=(4,1)$ $\xi^2=(0,2)$, $\xi^3=(1,0).$	
\begin{eqnarray*}
	\zeta_1(x,\xi^i)=x_1^2+(x_2-\xi^i_1)^2-\xi^i_2x_3^2~~i=1,2,3\\
	\zeta_2(x,\xi^i)=\xi^i_1x_1+\xi^i_2x_2^2+x_3+4\xi^i_1\xi^i_2~~i=1,2,3\\\mbox{and}~~~~
	\zeta_3(x,\xi^i)=\xi^i_1x_1^2+6x_2^2+25(x_3-\xi^i_2x_1)^2.
\end{eqnarray*}
%m=2,n=3,p=3,lb=(0,0,0)^T,ub=(1,1,1)^T
%m=2,n=1,p=2,lb=-3,ub=3
%m=2,n=2,p=2,lb=(-4,-4),ub=(4,4)
%m=2,n=3,p=3,lb=(-1,-2,-1)^T,ub=(1,1,2)^T
%m=3,n=3,p=3,lb=(1,-2,0)^T,ub=(3.5,2.0,1.0)^T
%m=2,n=2,p=2,lb=(-6,-6)^T,ub=(6,4)^T.
%m=3.n=3,p=3,lb=(-1,-1,-1),ub=(5,5,5).
	\item $F: \mathbb{R}^2\rightarrow \mathbb{R}^3$ is defined as $F_j(x)= \max\{\zeta_j(x,\xi^{i}): i=1,2,3\},$ $j=1,2,3,$ where $\xi^i=(\xi^1_i,\xi^2_i)$ and $\xi^1=(2,3)$ $\xi^2=(4,5)$, $\xi^3=(2,0).$
\begin{eqnarray*}
	\zeta_1(x,\xi^i)= x_1^2+\xi^i_1x_2^4+\xi^i_1\xi^i_2x_1x_2,~~i=1,2,3\\
	\zeta_2(x,\xi^i)=5x_1^2+\xi^i_1x_2^2+\xi^i_2x_1^4x_2,~~i=1,2,3\\\mbox{and}~~~~
	f_3(x,\xi^i)=e^{-\xi^i_1x_1+\xi^i_2x_2}+x_1^2-\xi^i_1x_2^2,~~i=1,2,3.
\end{eqnarray*}
%m=2,n=3,p=3,lb=(0,0,0)^T,ub=(1,1,1)^T
%m=2,n=1,p=2,lb=-3,ub=3
%m=2,n=2,p=2,lb=(-4,-4),ub=(4,4)
%m=2,n=3,p=3,lb=(-1,-2,-1)^T,ub=(1,1,2)^T
%m=3,n=3,p=3,lb=(1,-2,0)^T,ub=(3.5,2.0,1.0)^T
%m=2,n=2,p=2,lb=(-6,-6)^T,ub=(6,4)^T.
%m=3.n=3,p=3,lb=(-1,-1,-1),ub=(5,5,5).
%m=3.n=2,p=3,lb=(-1,-1,-1),ub=(5,5,5).
\item 	$F: \mathbb{R}\rightarrow \mathbb{R}^2$ is defined as $F_j(x)= \max\{\zeta_j(x,\xi^{i}): i=1,2\},$ $j=1,2,$ where $\xi^1=-3$ $\xi^2=8.$
\begin{eqnarray*}
	\zeta_1(x,\xi^i)=(x-\xi^i)^2,~~i=1,2\\\mbox{and}~~
\zeta_2(x,\xi^i)=-x^2-\xi^i x.
\end{eqnarray*}
%m=2,n=3,p=3,lb=(0,0,0)^T,ub=(1,1,1)^T
%m=2,n=1,p=2,lb=-3,ub=3
%m=2,n=2,p=2,lb=(-4,-4),ub=(4,4)
%m=2,n=3,p=3,lb=(-1,-2,-1)^T,ub=(1,1,2)^T
%m=3,n=3,p=3,lb=(1,-2,0)^T,ub=(3.5,2.0,1.0)^T
%m=2,n=2,p=2,lb=(-6,-6)^T,ub=(6,4)^T.
%m=3.n=3,p=3,lb=(-1,-1,-1),ub=(5,5,5).
%m=3.n=2,p=3,lb=(-1,-1,-1),ub=(5,5,5).
%m=2,n=1,lb=-100,ub=100.
\item $F: \mathbb{R}^2\rightarrow \mathbb{R}^2$ is defined as $F_j(x)= \max\{\zeta_j(x,\xi^{i}): i=1,2\},$ $j=1,2,$ where $\xi^i=(\xi^i_1,\xi^i_2)$ and $\xi^1=(1,1)$ $\xi^2=(0,2).$
\begin{eqnarray*}
	\zeta_1(x,\xi^i)=(x_1-\xi^i_1)^2+(x_2+\xi^i_2)^2,~~i=1,~2\\\mbox{and}~~
	\zeta_2(x,\xi^i)=(\xi^i_1x_1+\xi^i_2x_2)^2~~i=1,~2.
\end{eqnarray*}
%m=2,n=3,p=3,lb=(0,0,0)^T,ub=(1,1,1)^T
%m=2,n=1,p=2,lb=-3,ub=3
%m=2,n=2,p=2,lb=(-4,-4),ub=(4,4)
%m=2,n=3,p=3,lb=(-1,-2,-1)^T,ub=(1,1,2)^T
%m=3,n=3,p=3,lb=(1,-2,0)^T,ub=(3.5,2.0,1.0)^T
%m=2,n=2,p=2,lb=(-6,-6)^T,ub=(6,4)^T.
%m=3.n=3,p=3,lb=(-1,-1,-1),ub=(5,5,5).
%m=3.n=2,p=3,lb=(-1,-1,-1),ub=(5,5,5).
%m=2,n=1,lb=-100,ub=100.
%m=2,n=2,p=2,lb=(-2,-2),ub=(5,5)
\item $F: \mathbb{R}^3\rightarrow \mathbb{R}^3$ is defined as $F_j(x)= \max\{\zeta_j(x,\xi^{i}): i=1,2,3\},$ $j=1,2,3,$ where $\xi^i=(\xi^i_1,\xi^i_2)$ and $\xi^1=(5,4)$ $\xi^2=(0,8),$ $\xi^3=(4,0).$
\begin{eqnarray*}
	\zeta_1(x,\xi^i)=3x_1^2+(x_2-\xi^i_1)^2+\xi^i_2x_3^2~~~i=1,2,3\\
	\zeta_2(x,\xi^i)=2\xi^i_1x_1+\xi^i_2x_2^2+3x_3+4\xi^i_1\xi^i_2~~~i=1,2,3\\\mbox{and}~	\zeta_3(x,\xi^i)=\xi^i_1x_1^2+6x_2^2+20(x_3-\xi^i_2x_1)^2)~~~i=1,2,3.\\
\end{eqnarray*}
%m=2,n=3,p=3,lb=(0,0,0)^T,ub=(1,1,1)^T
%m=2,n=1,p=2,lb=-3,ub=3
%m=2,n=2,p=2,lb=(-4,-4),ub=(4,4)
%m=2,n=3,p=3,lb=(-1,-2,-1)^T,ub=(1,1,2)^T
%m=3,n=3,p=3,lb=(1,-2,0)^T,ub=(3.5,2.0,1.0)^T
%m=2,n=2,p=2,lb=(-6,-6)^T,ub=(6,4)^T.
%m=3.n=3,p=3,lb=(-1,-1,-1),ub=(5,5,5).
%m=3.n=2,p=3,lb=(-1,-1,-1),ub=(5,5,5).
%m=2,n=1,lb=-100,ub=100.
%m=2,n=2,p=2,lb=(-2,-2),ub=(5,5)
%m=3,n=3,p=3,lb=(0,0,0),ub=(1,1,1)
	\item \label{cec} $F: \mathbb{R}^2\rightarrow \mathbb{R}^3$ is defined as $F_j(x)= \max\{\zeta_j(x,\xi^{i}): i=1,2,3\},$ $j=1,2,3,$ where $\xi^1=(3.0,3.0,\dots,3.0)^T$, $\xi^2=(6.0,6.0,\dots,6.0)^T$, and $\xi^3=(9.0,9.0,\dot,9.0)^T$. Define
\begin{eqnarray*}
	I_1&=&\{k\in\{2,3,\dots,n\}| k ~~mod~~2=1\}\\
	I_1&=&\{k\in\{2,3,\dots,n\}| k ~~mod~~2=0\}\\
	\zeta_1(x,\xi^i)&=& x_1+\frac{2}{|I_1|} \sum_{k\in I_1} \left(x_k-\sin(\xi^i_k\pi x_1+\frac{k\pi}{n})\right)^2~~i=1,2,3\\
	\zeta_2(x,\xi^i)&=&1-\sqrt{x_1} +\frac{2}{I_2}\sum_{k\in I_2} \left(x_k-\sin(\xi^i_k\pi x_1+\frac{k\pi}{n})\right)^2~~i=1,2,3.
\end{eqnarray*}
%m=2,n=3,p=3,lb=(0,0,0)^T,ub=(1,1,1)^T
%m=2,n=1,p=2,lb=-3,ub=3
%m=2,n=2,p=2,lb=(-4,-4),ub=(4,4)
%m=2,n=3,p=3,lb=(-1,-2,-1)^T,ub=(1,1,2)^T
%m=3,n=3,p=3,lb=(1,-2,0)^T,ub=(3.5,2.0,1.0)^T
%m=2,n=2,p=2,lb=(-6,-6)^T,ub=(6,4)^T.
%m=3.n=3,p=3,lb=(-1,-1,-1),ub=(5,5,5).
%m=3.n=2,p=3,lb=(-1,-1,-1),ub=(5,5,5).
%m=2,n=1,lb=-100,ub=100.
%m=2,n=2,p=2,lb=(-2,-2),ub=(5,5)
%m=3,n=3,p=3,lb=(0,0,0),ub=(1,1,1)
%m=3,n=2,p=3,lb=(-4,-4),ub=(5,5)
{\bf Note:} Test problem TP\ref{cec} is formulated using CEC09\_1 in \citet{zhang2008multiobjective}.
	\item $F: \mathbb{R}^2\rightarrow \mathbb{R}^2$ is defined by $F_j(x)= \max\{\zeta_j(x,\xi^{i}): i=1,2\},$ $j=1,2.$
\begin{eqnarray*}
	\zeta_1=	\max\{\zeta_{1}(x,\xi^{1}),\zeta_{1}(x,\xi^{2})\}=\max\{1+x_1^{\frac{1}{4}},1+x_2^2\}\\
	F_2=	\max\{\zeta_{2}(x,\xi^{1}),\zeta_{2}(x,\xi^{2})\}=\bigg\{1-\bigg(\frac{x_1}{1+x_1^{\frac{1}{4}}}\bigg)^2,1-\bigg(\frac{x_1}{1+x_2}\bigg)^2\bigg\}
\end{eqnarray*}
%m=2,n=3,p=3,lb=(0,0,0)^T,ub=(1,1,1)^T
%m=2,n=1,p=2,lb=-3,ub=3
%m=2,n=2,p=2,lb=(-4,-4),ub=(4,4)
%m=2,n=3,p=3,lb=(-1,-2,-1)^T,ub=(1,1,2)^T
%m=3,n=3,p=3,lb=(1,-2,0)^T,ub=(3.5,2.0,1.0)^T
%m=2,n=2,p=2,lb=(-6,-6)^T,ub=(6,4)^T.
%m=3.n=3,p=3,lb=(-1,-1,-1),ub=(5,5,5).
%m=3.n=2,p=3,lb=(-1,-1,-1),ub=(5,5,5).
%m=2,n=1,lb=-100,ub=100.
%m=2,n=2,p=2,lb=(-2,-2),ub=(5,5)
%m=3,n=3,p=3,lb=(0,0,0),ub=(1,1,1)
%m=3,n=2,p=3,lb=(-4,-4),ub=(5,5)
%m=2,n=2,p=2,lb=(.01,.001),ub=(1,1).
	\item \label{dt1}  $F: \mathbb{R}^n\rightarrow \mathbb{R}^m$ is defined by $F_j(x)= \max\{\zeta_j(x,\xi^{i}): i=1,2,3\},$ $j=1,2,\dots,m,$ where $\xi^1=(0.25,0.25,\dots,0.25)^T\in \mathbb{R}^n$,  $\xi^2=(0.5,0.5,\dots,0.5)^T\in \mathbb{R}^n$, $\xi^3=(0.75,0.75,\dots,0.75)^T\in \mathbb{R}^n$. Define $$gx_i=100\left(K+\sum_{k=m}^n\left((x_i-\xi^i_k)^2-\cos(20\pi(x_k-\xi^i_k))\right)\right)$$
for $i=1,2,3$ where $K=m+n-1$. $\zeta_j(x,\xi^i) $ for $j=1,2,..., m$ are defined as
\begin{eqnarray*}
	\zeta_1(x,\xi^i)&=&0.5(1+gx_i)\prod_{k=1}^{m-1} x_k\\
	\zeta_j(x,\xi^i)&=&	0.5(1+gx_i)\prod_{k=1}^{m-j} x_k(1-x_{m-j+1}) ~~j=2,\dots,m
\end{eqnarray*}
%m=2,n=3,p=3,lb=(0,0,0)^T,ub=(1,1,1)^T
%m=2,n=1,p=2,lb=-3,ub=3
%m=2,n=2,p=2,lb=(-4,-4),ub=(4,4)
%m=2,n=3,p=3,lb=(-1,-2,-1)^T,ub=(1,1,2)^T
%m=3,n=3,p=3,lb=(1,-2,0)^T,ub=(3.5,2.0,1.0)^T
%m=2,n=2,p=2,lb=(-6,-6)^T,ub=(6,4)^T.
%m=3.n=3,p=3,lb=(-1,-1,-1),ub=(5,5,5).
%m=3.n=2,p=3,lb=(-1,-1,-1),ub=(5,5,5).
%m=2,n=1,lb=-100,ub=100.
%m=2,n=2,p=2,lb=(-2,-2),ub=(5,5)
%m=3,n=3,p=3,lb=(0,0,0),ub=(1,1,1)
%m=3,n=2,p=3,lb=(-4,-4),ub=(5,5)
%m=2,n=2,p=2,lb=(.01,.001),ub=(1,1).
	%m=2,n=len(x),p=2,lb=(.001,.001)ub=(1,1)
\item \label{dt2} $F: \mathbb{R}^n\rightarrow \mathbb{R}^m$ is defined by $F_j(x)= \max\{\zeta_j(x,\xi^{i}): i=1,2,3\},$ $j=1,2,\dots,m,$ where $\xi^1=(0.4,0.4,\dots,0.4)^T\in \mathbb{R}^n$,  $\xi^2=(0.5,0.5,\dots,0.5)^T\in \mathbb{R}^n$, $\xi^3=(0.6,0.6,\dots,0.6)^T\in \mathbb{R}^n$. Define $$gx_i=\sum_{k=m}^n(x_i-\xi^i_k)^2$$
for $i=1,2,3$. $f_j(x,\xi^i) $ for $j=1,2,..., m$ are defined as
\begin{eqnarray*}
	\zeta_1(x,\xi^i)&=&(1+gx_i)\prod_{k=1}^{m-1} \cos(0.5\pi x_k)\\
	\zeta_j(x,\xi^i)&=&	0.5(1+gx_i)\prod_{k=1}^{m-j}\cos(0.5\pi x_k)\sin(0.5\pi x_{m-k+1}) ~~j=2,\dots,m
\end{eqnarray*}
%m=2,n=3,p=3,lb=(0,0,0)^T,ub=(1,1,1)^T
%m=2,n=1,p=2,lb=-3,ub=3
%m=2,n=2,p=2,lb=(-4,-4),ub=(4,4)
%m=2,n=3,p=3,lb=(-1,-2,-1)^T,ub=(1,1,2)^T
%m=3,n=3,p=3,lb=(1,-2,0)^T,ub=(3.5,2.0,1.0)^T
%m=2,n=2,p=2,lb=(-6,-6)^T,ub=(6,4)^T.
%m=3.n=3,p=3,lb=(-1,-1,-1),ub=(5,5,5).
%m=3.n=2,p=3,lb=(-1,-1,-1),ub=(5,5,5).
%m=2,n=1,lb=-100,ub=100.
%m=2,n=2,p=2,lb=(-2,-2),ub=(5,5)
%m=3,n=3,p=3,lb=(0,0,0),ub=(1,1,1)
%m=3,n=2,p=3,lb=(-4,-4),ub=(5,5)
%m=2,n=2,p=2,lb=(.01,.001),ub=(1,1).
%m=2,n=len(x),p=2,lb=(.001,.001)ub=(1,1)
%m=3,n=len(x),p=3,lb=(0.001,..,0.001),ub=(1,1...,1).
{\bf Note:} TP\ref{dt1} and TP\ref{dt2} are are constructed using the test problems DTLZ1 and DTLZ2 respectively from \citet{deb2002scalable}.
\end{enumerate}
\end{appendices}
\end{document}